\documentclass[reqno]{amsart}

\usepackage{amsmath}
\usepackage{amsthm}
\usepackage{amsfonts}

\newtheorem{thm}{Theorem}[section]
\newtheorem{lem}[thm]{Lemma}
\newtheorem{cor}[thm]{Corollary}
\newtheorem{prop}[thm]{Proposition}

\newtheorem{rem}[thm]{Remark}

\newtheorem{deff}[thm]{Definition}

\numberwithin{equation}{section}

\newcommand{\R}{\mathbb{R}}

\newcommand{\N}{\mathbb{N}}

\newcommand{\mcm}{\mathcal{M}}
\newcommand{\mcd}{\mathcal{D}}

\newcommand{\Om}{\Omega}

\newcommand{\vp}{\varphi}

\newcommand{\rd}{\mathrm{d}}
\newcommand{\divv}{\mathrm{div}}

\newcommand{\Ia}{I(\alpha)}
\newcommand{\li}{\lambda_{i,\alpha}}
\newcommand{\bi}{b_{i,\alpha}}
\newcommand{\mi}{\mu_{i,\alpha}}
\newcommand{\lis}{\lambda_{i,\alpha}^*}
\newcommand{\bis}{b_{i,\alpha}^*}
\newcommand{\bqn}{\begin{equation}}
\newcommand{\eqn}{\end{equation}}
\newcommand{\bqnn}{\begin{equation*}}
\newcommand{\eqnn}{\end{equation*}}
\newcommand{\bear}{\begin{eqnarray}} 
\newcommand{\eear}{\end{eqnarray}} 
\newcommand{\bean}{\begin{eqnarray*}} 
\newcommand{\eean}{\end{eqnarray*}} 
\newcommand{\bs}{\begin{split}}
\newcommand{\es}{\end{split}}

\title[Proteus Mirabilis Swarm-Colony Development with Drift]
{{\it Proteus Mirabilis} Swarm-Colony Development with Drift}

\author[Ph. Lauren\c{c}ot and Ch. Walker]
{Philippe Lauren\c{c}ot and Christoph Walker}

\address{Institut de Math\'{e}matiques de Toulouse, CNRS UMR~5219, Universit\'{e} de Toulouse, F--31062 Toulouse cedex 9, France.}
\email{Philippe.Laurencot@math.univ-toulouse.fr}

\address{Leibniz Universit\"at Hannover, Institut f\"ur Angewandte Mathematik, Welfengarten 1, D--30167 Hannover, Germany.}
\email{walker@ifam.uni-hannover.de}

\begin{document}

\begin{abstract}
We prove a global existence result for  a model describing the swarming phenomenon of the bacterium {\it Proteus mirabilis}. The model consists of an ordinary differential equation coupled with an age-structured equation involving nonlinear degenerate diffusion and an additional drift term.
\end{abstract}


\maketitle

\section{Introduction}

{\it Proteus mirabilis} is a bacterium that is widely distributed
in soil and water in the natural environment. It is also found in
the intestinal tract of many mammals, including human. Broth
cultures of {\it Proteus mirabilis} consist of small {\it swimmer cells}, but produce a morphologically and physiologically distinct cell type, called {\it swarmer cells}, when inoculated on a solid surface. This
process is referred to as ``differentiation'' and is crucial for
the pathogenesis of these bacteria during urinary tract infections
caused, e.g., by long-term urinary catheterization. While swimmer cells go through a prototypical cell division process and are immobile, swarmer cells age and increase in size. Swarmer cells can group together to build multicellular ``rafts'' that, when of sufficient biomass, are capable of translocation. This leads to a migration phase during which swarmer cells may also dedifferentiate again into swimmer cells. Once the biomass falls below the critical threshold, movement ceases initiating a consolidation phase. This oscillation between phases of motion of swarmer cells and consolidation to the swimmer state leads to an interesting bull's-eye-patterned biofilm.

Different mathematical models were proposed in order to
describe the swarming of {\it Proteus mirabilis}. In this article
we focus on a model that was presented and numerically analyzed in
\cite{EsipovShapiro} and later, in slightly modified form, in \cite{Ayati1, MedvedevEtal}. The model involves the swimmer  cell density $v=v(t,x)$ in dependence of time $t$ and spatial position $x$ and the swarmer cell density $u=u(t,a,x)$, where the variable $a$
models cell age. The equations under consideration are
\begin{align}
\partial_t u+\partial_a u&\,=\, \divv_x \big(D(\Lambda)\,\nabla_x u\,
+\, u\, E(\Lambda,v)\, \nabla_x\Lambda \big)\, -\,\mu(a)\, u\ ,  & (t,a,x)\in (0,\infty)^2\times\Om\ ,\label{1}\\
\partial_t v&\,= \,\big(g(v)\,-\,\xi(v)\big)\, v\,+\int_0^\infty
b(a)\,\mu(a)\, u(t,a,x)\, \rd a\ ,  &(t,x)\in (0,\infty)\times\Om\
,\label{2}
\end{align}
where
    \bqn\label{3}
    \Lambda(t,x)\,:=\int_0^\infty \lambda(a)\, u(t,a,x)\, \rd a\ ,\quad (t,x)\in (0,\infty)\times\Om\ ,
    \eqn
subject to the boundary conditions
\begin{align}
u(t,0,x)&\,=\, \xi(v(t,x))\,v(t,x)\ ,\quad &(t,x)\in(0,\infty)\times\Om\ ,\label{4}\\
D(\Lambda)\,\partial_\nu\, u\,+\,u\,E(\Lambda,v)\,\partial_\nu
\Lambda&\,=\,0\ , &(t,a,x)\in (0,\infty)^2\times \partial\Om\
,\label{5}
\end{align}
and the initial conditions
    \bqn\label{6}
 u(0,a,x)=u^0(a,x)\ ,\quad
v(0,x)=v^0(x)\ ,\qquad (a,x)\in (0,\infty)\times\Om\ .\\
    \eqn

\noindent Here, $\Om\subset\R^N$ is an open and bounded set with smooth boundary $\partial\Om$ and $\nu=\left( \nu^1,\ldots,\nu^N \right)$ denotes the outward normal unit vector field to $\partial\Om$. 
The function $\Lambda$ given by \eqref{3} represents the total motile swarmer cell biomass. Since increase in size of swarmer cells is exponential with increase in age, the function $\lambda$ appearing in the definition of $\Lambda$ is often taken in the
form
$$
\lambda(a)\, =\, m_0\, {\bf 1}_{[a_{0},\infty)}(a)\, e^{a/\tau}\ ,
$$
where $a_{0}\ge 0$ is the minimal age of swarmer cells required
to participate actively in group migration. The parameter $\tau$ is the average time it takes a cell to subdivide, and $m_0>0$ is a constant. 

Equation \eqref{1} expresses the change in time of swarmer cells
of a given age $a$. The diffusivity $D$ depends on $\Lambda$ and
is zero for $\Lambda=0$ or, more generally, for $\Lambda$ small.
The explicit appearance of the ``drift'' term $\divv_x\big(u
E(\Lambda,v)\nabla_x \Lambda\big)$ on the right-hand side of
\eqref{1} is a novelty, although it has been already implicitly
contemplated in the model of \cite{Ayati1}. More precisely, in
\cite{Ayati1} the diffusion term is derived from isotropic random motion (instead of Fickian diffusion) and thus written in radially
symmetric coordinates as $(1/r) \partial_r(r\partial_r(D_A(\Lambda)u))$ which corresponds to the choice $D=D_A$ and $E(\Lambda,v)=D_A'(\Lambda)$ in our case. As pointed out in \cite{Ayati1, EsipovShapiro} only swarmer cells of a certain maturity can actively participate in group migration but nothing prevents young swarmers from being caught up in the flow and thus move with larger swarmers in the rafts. However, diffusion terms of the form $\divv_x(D(\Lambda)\nabla_x u)$ as considered in \cite{EsipovShapiro, Frenod, MedvedevEtal} reflect active movement
of swarmers of any age, i.e. also of young swarmers. Therefore, we
hypothesize that this term could be rather small (or even zero) and
that migration of swarmers could be mainly due to the drift term
$\divv_x\big(u E(\Lambda,v)\nabla_x \Lambda\big)$ in which small
(i.e. young) swarmers move but do not actively contribute to a raft's motility. Interesting would be, of course, to see whether numerical computations can support this hypothesis. The age dependent function $\mu$ in \eqref{1} is the dedifferentiation modulus, which is higher for older swarmers than for younger ones.

The change in time of the swimmer population is given by equation
\eqref{2}. The population grows exponentially, where often $g(v)=\tau^{-1}$ in numerical simulations. Swimmer cells differentiate
with rate $\xi(v)$ into swarmers of age $0$ leading to the age
boundary condition \eqref{4}. Usually, $\xi$ is of the form
$\xi(v)=\tilde{\xi}(v)/\tau$, where $\tilde{\xi}(v)=0$
for small and large values of $v$, respectively. The incorporation of a lag phase in swarmer cell production triggers the development of the consolidation phase after a swarm phase. It thus prevents a self-sustaining soliton caused by swarmers that dedifferentiate into swimmers immediately
differentiating into new swarmers. This lag in the onset of
differentiation was used in \cite{Ayati1,Ayati2,MedvedevEtal}. The integral term in \eqref{2} represents dedifferentiation of swarmer cells
into swimmer cells. In the numerical simulations in \cite{Ayati1, Ayati2, EsipovShapiro, MedvedevEtal} the function $b$ is given by $b(a)=e^{a/\tau}$.

As for further explanation of the model and for computational
results regarding \eqref{1}-\eqref{6} we refer to \cite{Ayati1, Ayati2, EsipovShapiro, Frenod2, MedvedevEtal} and the references therein. Existence results for \eqref{1}-\eqref{6} in the case of non-degenerate diffusion and $E\equiv 0$ can be found in \cite{Frenod}, while degenerate diffusion and $E\equiv 0$ was studied in \cite{MMNP}. 

The purpose of this article is to investigate mathematically
equations \eqref{1}-\eqref{6} for degenerate diffusion with
non-vanishing drift term. As already pointed out, the latter, in our opinion, could possibly be more important in the migration process than diffusion. It also generates additional difficulties in the mathematical analysis: indeed, while an $L_\infty$-estimate for $u$ is readily obtained from \eqref{1} when $E=0$, such a bound does not seem to be available in the presence of the drift term ($E\ne 0$) and we only obtain a much weaker $L \ln{L}$-estimate on $u$, see Lemma~\ref{L2.5} below. As a consequence of this lower regularity and the possible degeneracy of $D$ and $E$, the diffusion and drift terms are not well-defined in \eqref{1} and a weak formulation is required. The latter is introduced in Definition~\ref{D} and is somehow reminiscent of the definition of renormalized solutions for the Boltzmann equation \cite{DPL89} or parabolic equations (see, e.g., \cite{AmmarWittbold, BM97, BMR01} and the references therein).

Before stating precisely the assumptions on the data used in this paper and the results obtained, let us briefly outline the difficulties to be overcome and sketch our approach. First, the system \eqref{1}-\eqref{6} is of mixed type and features several nonlinearities. We will thus use a compactness method, that is, first establish the existence of solutions to a sequence of approximate problems and then pass to the limit as the approximation parameter converges to zero. Besides standard approximations (such as the positivity of $D$, the boundedness of $E$, and an additional linear diffusion in \eqref{2}), the approximation used herein relies upon the discretization of \eqref{1} with respect to the age variable which leads us to a system of parabolic equations to which the abstract theory developed by Amann \cite{Amann93} can be applied. Next, \eqref{1} is a first-order transport equation with respect to the age variable and a degenerate parabolic equation with respect to the space variable, while \eqref{2} features no spatial diffusion. The latter thus does not provide any smoothing effect which would guarantee the strong compactness for $v$ needed to pass to the limit in the nonlinear terms $E(\Lambda,v)$, $g(v)$, and $\xi(v)$. Strong compactness for $v$ can thus only result from that of the last term of the right-hand side of \eqref{2} which requires the strong compactness for the age averages of $u$. Such a compactness property can only be deduced from \eqref{1} but is hindered by the possible degeneracy of the diffusion term $D(\Lambda)\nabla_x u$ for small values of $\Lambda$. Nevertheless, it holds true under the assumptions that $\lambda$ is bounded from below by a positive constant while $D$ only vanishes when $\Lambda$ vanishes. It is, however, the main obstacle to include the case $D\equiv 0$ in our analysis, see Remark \ref{Miraculix}. As already mentioned, another difficulty to be faced is that the drift term $\divv_x(u E(\Lambda,v) \nabla_x\Lambda)$ only allows us to obtain an estimate of $u$ in $L \ln{L}$ while $E(\Lambda,v) \nabla_x\Lambda$ belongs merely to $L_2$, so that the drift term is not well-defined. Here, we take advantage of the $L_\infty$-boundedness of the age averages of $u$ to set up a weak formulation which complies with the available regularity of $u$. Concerning compactness estimates for $u$, they rely on the just mentioned $L \ln{L}$-estimate derived from the specific structure of \eqref{1} as well as a degenerate parabolic equation in the variables $t$ and $x$ satisfied by the age averages of $u$. While the former guarantees the weak compactness for $u$ in $L_1$, the latter provides the expected strong compactness on the age averages of $u$ thanks to our assumptions on the data.

We shall also remark that one can use different approaches to investigate the existence of solutions to age structured equations with (non-degenerate) diffusion, including integrated semigroups, perturbation arguments, or using solutions integrated along characteristics (e.g., see \cite{MagalThieme, RhandiSchnaubelt, WebbSpringer} and the references therein.) However, handling such equations by discretizing with respect to the age variable seems to be a novel approach. Let us also point out that the method used in \cite{MMNP} to tackle the case of age structure with quasi-linear non-degenerate diffusion is apparently not applicable in the present situation due to the additional drift term.

\begin{rem}\label{remMKK}
In the particular case where $\lambda(a)=m_0\ e^{a/\tau}$, $b(a) = m_1\ e^{a/\tau}$, and $\mu(a)=m_2$ for some positive real numbers $m_0$, $m_1$, $m_2$, and $\tau$, a closed system for the evolution of $\Lambda$ and $v$ can be derived from \eqref{1}-\eqref{2} and reads
\begin{eqnarray*}
\partial_t \Lambda &\,=\,& \divv_x \big( \left( D(\Lambda) \,+\, \Lambda\, E(\Lambda,v) \big)\, \nabla_x\Lambda \right)\, + \left( \frac{1}{\tau} \,-\,m_2 \right)\, \Lambda\ , \quad (t,x)\in (0,\infty)\times\Om\ ,\\
\partial_t v &\,=\,& \big(g(v)\,-\,\xi(v)\big)\, v\,+ \frac{m_1\, m_2}{m_0}\, \Lambda\ , \quad (t,x)\in (0,\infty)\times\Om\ .
\end{eqnarray*}
The analysis performed in \cite{MedvedevEtal} actually focuses on this ``reduced'' system with the choice of $$D(\Lambda)\,+\, \Lambda\, E(\Lambda,v)\,=\,\frac{D_0 \Lambda}{\Lambda+kv}\ ,
$$
where $D_0,k>0$.
\end{rem}

\section{Existence}

Regarding the data in \eqref{1}-\eqref{6} we will assume that the following hypotheses hold:

\begin{itemize}
\item[$(h_1)$] The functions $\xi, g\in L_\infty(\R)\cap \mathcal{C}^1(\R)$ are such that $\xi(s)=0$ for $s\le 0$ and $ 0\le \xi(s)\le g(s)$ for $s\ge 0$.

\item[$(h_2)$] The function $b\in \mathcal{C}^1([0,\infty))$ is
non-decreasing with $b(0)=1$ and $b(a)\to\infty$ as $a\to\infty$, 
and there exists a number $B_0\in (0,\infty)$ such that
$$
   b(a+\alpha)\,\le\, (B_0\,\alpha\, +\,1)\, b(a)\ ,\qquad a> 0\ ,\quad \alpha\in (0,1)\ .
$$

\item[$(h_3)$] The function $\lambda\in \mathcal{C}^1([0,\infty))$ is non-negative, satisfies
$\ell_0:=\inf\lambda>0$, and there exists a number  $L_0\in(0,\infty)$
such that
$$
\begin{array}{ll}
 \lambda(a+\alpha)\,\le\, (L_0\,\alpha\, +\,1)\, \lambda(a)\ , & \qquad a> 0\ ,\quad \alpha\in (0,1)\ ,\\
  \lambda(a-\alpha)\,\le\, (L_0\,\alpha\, +\,1)\, \lambda(a)\ , & \qquad a> \alpha\ ,\quad \alpha\in (0,1)\ .
\end{array}
$$

\item[$(h_4)$] The function $\mu\in L_\infty(0,\infty)$ is
non-negative, and there exists a number $\beta_0\in (1,\infty)$ such
that
$$
\mu(a)\,b(a)\,\le\, \beta_0\, \lambda(a)\,\le\, \beta_0^2\, b(a)\ ,\quad a\ge 0\ .
$$

\item[$(h_5)$] The function $D\in \mathcal{C}^2(\R)$ is non-decreasing with $D(r)>0$ for $r>0$ and $[r\mapsto (D(r)/r)^{1/2}]\in L_1(0,1)$. Moreover, for the function $\zeta_1$, defined by $\zeta_1'(r):= (D(r)/r)^{1/2}$ and $\zeta_1(0)=0$, we assume that $D'/\zeta_1'\in\mathcal{C}([0,\infty))$ and put 
$$
\kappa_1(R):=\sup_{0\le r\le R}\frac{D'(r)}{\zeta_1'(r)}\ ,\quad R>0\ .
$$

\item[$(h_6)$] The function $E\in \mathcal{C}^3(\R^2)$ is non-negative, and there is a function $\zeta_2\in \mathcal{C}^1(\R)$ such that $\zeta_2(0)=0$, $\zeta_2'(r)>0$ for $r>0$, $E/\zeta_2'\in \mathcal{C}([0,\infty)\times [0,\infty))$, and 
$$
\kappa_2(R)\ \zeta_2'(r)^2 \le E(r,s) \le \kappa_3(R)\ \zeta_2'(r)\ , \quad (r,s) \in [0,R]\times [0,R]\ ,
$$
for some constants $\kappa_2(R)>0$, $\kappa_3(R)>0$ and all $R>0$.
\end{itemize}

\noindent Note that $(h_2)$ and $(h_3)$ are satisfied, e.g., by
$b(a)=\lambda(a)=e^{a/\tau}$, $a\ge 0$, with $\tau>0$, which is one of the choices of $\lambda$ and $b$ in \cite{Ayati1, Ayati2, EsipovShapiro, Frenod, MedvedevEtal}. Also note that the function $D$ in $(h_5)$ may be such that $D(0)=0$, that is, we may allow for a degeneracy of the diffusion coefficient $D(r)$ at $r=0$. For example, $D(r):=D_0r^\theta$ with $D_0>0$ and $\theta\ge 1$ (or $\theta=0$) satisfies $(h_5)$. Following \cite{Ayati2}, the function $E(r):=D'(r)=\theta D_0 r^{\theta-1}$ fulfills assumption $(h_6)$ in this case.

\begin{deff}\label{D} Suppose $(h_1)-(h_6)$. A (global) {\it weak solution} to \eqref{1}-\eqref{6} is a pair of non-negative functions $(u,v)$ possessing, for each $T>0$, the regularity
$$
u\in L_\infty(0,T;L_1((0,\infty)\times\Om; b(a) \rd a\rd x))\,, \quad v\in \mathcal{C}^1([0,T];L_\infty(\Om))\,,
$$
$$
\Lambda\,:=\int_0^\infty \lambda(a)\, u(.,a,.)\, \rd a\in L_\infty((0,T)\times\Om)\,, \quad \zeta_j(\Lambda)\in L_2(0,T;W_2^1(\Om))\,, \quad j=1,2\,, 
$$
and satisfying $v(0)=v^0$, 
$$
\partial_t v \,= \,\big(g(v)\,-\,\xi(v)\big)\, v\,+\int_0^\infty
b(a)\,\mu(a)\, u(.,a,.)\, \rd a\  \;\;\mbox{ a.e. in }\;\; (0,T)\times\Om\,,
$$
and
\bqnn
\begin{split}
0  = & \int_0^T \int_\Om \int_0^\infty \left( \partial_t \varphi + \partial_a \varphi - \mu\, \varphi \right)\ u\ \rd a \rd x \rd t + \int_0^T \int_\Om \varphi(t,0,x)\ \xi(v(t,x))\ v(t,x)\ \rd x \rd t \\
&\quad +  \int_\Om \int_0^\infty \varphi(0,a,x)\ u^0(a,x)\ \rd a \rd x + \int_0^T \int_\Om \int_0^\infty \Delta_x\varphi\ D(\Lambda)\ u\  \rd a \rd x \rd t \\
&\quad -  \int_0^T \int_\Om \left\{ \frac{E(\Lambda,v)}{\zeta_2'(\Lambda)}\ \nabla_x \zeta_2(\Lambda) - \frac{D'(\Lambda)}{\zeta_1'(\Lambda)}\ \nabla_x \zeta_1(\Lambda)\right\} \cdot \left( \int_0^\infty u\ \nabla_x\varphi\  \rd a\right)\ \rd x \rd t
\end{split}
\eqnn
for any test function $\varphi\in\mathcal{C}^2([0,T)\times(0,\infty)\times\bar{\Om})$ with compact support and $\partial_\nu \varphi(t,a,x)=0$ for $(t,a,x)\in (0,T)\times (0,\infty)\times\partial\Om$.
\end{deff}

Observe that the regularity required on $u$, $v$, and $\Lambda$ along with assumptions $(h_3)$, $(h_5)$, and $(h_6)$ ensure that all the terms in the weak formulation for \eqref{1} are meaningful. In particular, $E(\Lambda,v)/\zeta_2'(\Lambda)$ and $D'(\Lambda)/\zeta_1'(\Lambda)$ are both bounded by $(h_5)$ and $(h_6)$ while $(h_3)$ implies that
$$
\left|\int_0^\infty u\ \nabla_x\varphi\  \rd a \right| \le \frac{\|\nabla_x\varphi\|_\infty}{\ell}\ \Lambda \in L_\infty((0,T)\times\Om)\, .
$$
Let us emphasize here once more that only the averages of $u$ with respect to age have the needed integrability properties for the weak formulation to make sense. A related situation is encountered in the theory of renormalized solutions for the Boltzmann equation \cite{DPL89} and parabolic equations (see, e.g., \cite{AmmarWittbold, BM97,BMR01} and the references therein).

\medskip

Our main result is then the following:

\begin{thm}\label{T}
Suppose $(h_1)-(h_6)$ and let $p>N$. Then, given any non-negative
initial values $(u^0,v^0)$ satisfying
\bear
\label{id1}
& & u^0\in L_1\big((0,\infty)\times\Om ; b(a) \rd a\rd x\big)\cap L_1\big((0,\infty),W_p^1(\Om);\lambda(a) \rd a \big)\,, \quad v^0\in W_p^1(\Om)\,, \\
\label{id2}
& & u^0 \ln{u^0} \in L_1((0,\infty)\times\Om; \lambda(a) \rd a\rd x)\,,
\eear
there exists a global weak solution to \eqref{1}-\eqref{6}. 
\end{thm}

The regularity assumptions on the data $D$, $b$, $\lambda$, and the initial data $u^0$, $v^0$ could be weakened, see Remark~\ref{hinkelstein}.

\section{A Regularized Problem}\label{sect2}

The basic idea to handle age structure is to discretize equation
\eqref{1} with respect to the age variable $a\in (0,\infty)$. To
that end we fix $I\in\N$ and $\alpha\in (0,1)$. Let then
$$
(b_i)_{1\le i\le I+1}\ ,\quad (\lambda_i)_{1\le i\le I+1} \
,\quad\text{and} \quad (\mu_i)_{1\le i\le I+1}
$$
be non-negative numbers such that
    \bqn\label{8}
    \begin{array}{l}
    b_i\ge 1\,, \qquad \lambda_i\ge\ell>0\,,\\
    0\,\le\, b_i^*\,\le\, B\, b_i\ ,\qquad \lambda_i^*\,\le\, L\,
    \lambda_i\ ,\qquad \mu_i\,\le M\ ,\qquad \mu_i\, b_i\,\le\,\beta \,\lambda_i\,\le\, \beta^2\, b_i\ ,
    \end{array}
    \eqn
for some positive numbers $\ell, B, L, M$, and $\beta$, where we put
$$
b_i^*\,:=\,\dfrac{b_{i+1}-b_i}{\alpha} \ ,\qquad
\lambda_i^*\,:=\,\dfrac{\lambda_{i+1}-\lambda_i}{\alpha}\ ,\qquad
i=1,\ldots,I\ .
$$
Moreover, let $\Theta\in \mathcal{C}^\infty(\R)$ be a cut-off function
satisfying
    \bqn\label{Theta}
    \Theta\ge 0\ ,\qquad \Theta'\,\le\, 0\ ,\qquad \Theta(r)\,=\,
    1\quad\text{for}\quad r\le 1/2\qquad \text{and}\qquad
    \Theta(r)=0\quad\text{for}\quad r\ge 1\ .
    \eqn
Finally, let $\xi$, $g$, $D$, and $E$ be functions satisfying $(h_1)$, $(h_5)$, $(h_6)$ together with
    \bqn\label{17}
    D(r)\ge d_0>0\ ,\quad (1+s)\,\xi(s) + E(r,s) \le \Xi\,, \quad (r,s)\in \R^2 \,,
    \eqn
for some constants $d_0>0$ and $\Xi>0$.

\medskip

We then look for a solution $(u_1,\ldots,u_I,\Lambda,v)$ to
the approximating problem
    \begin{align}
&    \partial_t u_i+\frac{1}{\alpha}(u_i-u_{i-1})= \divv_x\left(
    D(\Lambda)\nabla_x u_i+ u_i \Theta(\alpha^2 u_i)\, E(\Lambda,v)\,              \nabla_x\Lambda)\right)-\mu_i u_i\ ,\label{10}\\
&    \partial_t\Lambda=
    \divv_x\left(\left[D(\Lambda)+\sum_{i=1}^I \alpha
    \lambda_i u_i E(\Lambda,v) \right]\nabla_x \Lambda\right)
    +\lambda_1 \xi(v)v +\sum_{i=1}^{I}
    \alpha(\lambda_i^*-\mu_i\lambda_i)u_i - \lambda_{I+1}\ u_I\ ,\label{12}\\
&    \partial_t v= \alpha \Delta_x v+\big(g(v)-\xi(v)\big)v +\alpha
    \sum_{i=1}^I b_i \mu_i u_i\label{13}
    \end{align}
for $i=1,\ldots,I$ and $(t,x)\in (0,\infty)\times\Om$, where
    \bqn\label{14}
    u_0\,:=\, \xi(v)v\ ,
    \eqn
and subject to
    \bqn\label{15}
   \partial_\nu u_i=\partial_\nu \Lambda = \partial_\nu v=0\ ,\qquad (t,x)\in (0,\infty)\times\partial\Om\ ,
    \eqn
and
    \bqn\label{16}
    u_i(0,x)=u_i^0(x)\ ,\quad
    \Lambda(0,x)=\Lambda^0(x):=\alpha \sum_{j=1}^I \lambda_j
    u_j^0(x)\ , \quad v(0,x)=v^0(x)\ ,\qquad x\in\Om\ ,
    \eqn
for $ i=1,\ldots,I$. We will first show that
\eqref{10}-\eqref{16} possesses a classical solution and then
derive some uniform bounds on this solution.

\subsection{Global Existence}

In this subsection we prove the global well-posedness of
\eqref{10}-\eqref{16}. More precisely, we have the following
result.

\begin{prop}\label{P2.1}
Let $p>N$, $\alpha\in (0,1)$, and consider
$(u_1^0,\ldots,u_I^0,v^0)\in W_p^1(\Om,\R^{I+1})$ with
    \bqn\label{18}
    0\le u_i^0(x)\le \frac{1}{4\,\alpha^2}\ ,\qquad 0\le v^0(x)\ ,\qquad x\in \Om\ .
    \eqn
Then there exists a unique classical solution
$$
(u_1,\ldots,u_I,\Lambda,v)\in
\mathcal{C}\big([0,\infty)\times\bar{\Om},\R^{I+2}\big)\cap
\mathcal{C}^{1,2}\big((0,\infty)\times\bar{\Om},\R^{I+2}\big)
$$
to problem \eqref{10}-\eqref{16} such that $u_i(t,x)\ge 0$ and
$v(t,x)\ge 0$ for $i=1,\ldots,I$ and $(t,x)\in
[0,\infty)\times\bar{\Om}$. Moreover, setting
    \bqn\label{19}
    t_\alpha^*:=\sup\left\{t>0\,;\, \max_{1\le i\le I} \sup_{\tau\in [0,t]} \|u_i(\tau)\|_\infty  \le
    \frac{1}{2\alpha^2}\right\}\,>0\ ,
    \eqn
we have
    \bqn\label{20}
    \Lambda(t,x)=\alpha\sum_{i=1}^I \lambda_i u_i(t,x)\ ,\quad
    (t,x)\in [0,t_\alpha^*)\times\bar{\Om}\ .
    \eqn
\end{prop}

\begin{proof}
We fix $\eta$ such that
$$0<\eta<\frac{d_0}{\alpha\, \|E\|_\infty\, I\,\max\limits_{1\le i\le
I}\lambda_i}
$$
and put $D_0:=(-\eta,\infty)^{I+2}$. Moreover, we define
$a:=(a^{m,n})\in \mathcal{C}^2\big(D_0,\mathcal{L}(\R^{I+2})\big)$ by
$$
a(y)\, :=\,\left(
  \begin{array}{c c c c c c}
    D(y_{I+1})&  0 & \ldots & 0 & y_1\Theta(\alpha^2 y_1)E(y_{I+1},y_{I+2}) & 0\\
    0   & D(y_{I+1}) & \ddots & \vdots & \vdots & \vdots\\
     \vdots  & & \ddots & 0 & \vdots&  \\
     \vdots  &  & \ddots & D(y_{I+1})& y_I\Theta(\alpha^2
     y_I)E(y_{I+1},y_{I+2})&\vdots  \\
     \vdots & & & 0 &
     D(y_{I+1})+\alpha\sum\limits_{i=1}^I \lambda_i
     E(y_{I+1},y_{I+2}) y_i & 0\\
     0 &\ldots & \ldots & 0 & 0& \alpha
  \end{array} \right)
$$
for $y=(y_1,\ldots,y_{I+2})\in D_0$. We next set $a_{j,k}(y):=
a(y)\delta_{j,k}$ for $1\le j,k\le N$ and $y\in D_0$, and we
introduce the operators
$$
    \mathcal{A}(y) z:= -\sum_{j,k=1}^N
    \partial_j\big(a_{j,k}(y)\partial_k z\big)\ ,\qquad
    \mathcal{B}(y) z:= \sum_{j,k=1}^N \nu^j\, a_{j,k}(y)\partial_k
    z
$$
for $z=(z_1,\ldots,z_{I+2})$ and the function
$$
f(y)\, :=\,\big(f^m(y)\big)_{1\le m\le I+2}\,:=\,\left(
  \begin{array}{c}
  \displaystyle{ -\mu_1 y_1 -\frac{1}{\alpha}\left(y_1-\xi(y_{I+2}) y_{I+2}\right)}\\
  \\
  \displaystyle{ -\mu_2 y_2 -\frac{1}{\alpha} (y_2-y_1)}\\
  \\
   \vdots\\
   \\
  \displaystyle{ -\mu_I y_I -\frac{1}{\alpha} (y_I-y_{I-1})}\\
  \\
   \displaystyle{\lambda_1\xi(y_{I+2})y_{I+2}+\alpha\sum\limits_{i=1}^I
   (\lambda_i^*-\mu_i\lambda_i)y_i - \lambda_{I+1} y_I}\\
   \\
   \displaystyle{(g-\xi)(y_{I+2})y_{I+2} +\alpha \sum\limits_{i=1}^I b_i \mu_i y_i}
  \end{array} \right) \ .
$$
With these notations, an abstract formulation of
\eqref{10}-\eqref{16} reads
    \begin{align*}
    \partial_t z+\mathcal{A}(z)z&= f(z)\ ,\\
    \mathcal{B}(z)z&=0\ ,\\
    z(0)&=(u_1^0,\ldots,u_I^0,\Lambda^0,v^0)\ .
    \end{align*}
Clearly, owing to \eqref{17} and the choice of $\eta$, the
eigenvalues of $a(y)$ are positive for each $y\in D_0$, and the
boundary-value operator $(\mathcal{A},\mathcal{B})$ is of
separated divergence form in the sense of
\cite[Ex.~4.3(e)]{Amann93}. Consequently, the boundary-value
operator $(\mathcal{A},\mathcal{B})$ is normally elliptic. It then
follows from \cite[Thm.14.4, Thm.14.6]{Amann93} that
\eqref{10}-\eqref{16} has a unique maximal classical solution
$$
z= (u_1,\ldots, u_I, \Lambda,v)\in
\mathcal{C}\big([0,t^+)\times\bar{\Om},D_0\big)\cap
\mathcal{C}^{1,2}\big((0,t^+)\times\bar{\Om},\R^{I+2}\big)\ ,
$$
where $t^+\in (0,\infty]$ denotes the maximal time of existence.
Observe that $a^{1,n}(0,y_2,\ldots,y_{I+2})=0$ for
$n\in\{2,\ldots,I+2\}$ and $f^1(0,y_2,\ldots,y_{I+2})\ge 0$ by
$(h_1)$. Therefore, \cite[Thm.15.1]{Amann93} ensures that
$u_1(t,x)\ge 0$ for $(t,x)\in [0,t^+)\times\bar{\Om}$. Now, $u_2$
solves
$$
\partial_t u_2 -\divv_x\big(D(\Lambda)\nabla_x u_2+u_2
\Theta(\alpha^2 u_2) E(\Lambda,v) \nabla_x \Lambda \big)
+\left( \mu_2+\frac{1}{\alpha} \right)u_2\,=\, \frac{1}{\alpha} u_1\,\ge\, 0
$$
with a non-negative initial condition, which readily entails that
$u_2(t,x)\ge 0$ for $(t,x)\in [0,t^+)\times\bar{\Om}$ by the
comparison principle. Proceeding by induction, we obtain in a
similar way that $u_i(t,x)\ge 0$ for $(t,x)\in
[0,t^+)\times\bar{\Om}$ and $i\in \{1,\ldots,I\}$. The same
argument gives $v(t,x)\ge 0$ for $(t,x)\in
[0,t^+)\times\bar{\Om}$.

We next show that $t^+=\infty$. To that end we define the
parabolic operator $\mathcal{L}_1$ by
    \bqnn
    \begin{split}
    \mathcal{L}_1\,w \,:=\, \partial_t
    &w\,-\,\divv_x\big(D(\Lambda)\nabla_x w\big)\, -\, \big(\alpha^2
    u_1\Theta'(\alpha^2 u_1)\,+\,\Theta(\alpha^2 u_1)\big) 
    E(\Lambda,v)\nabla_x\Lambda \cdot \nabla_x w\\
    &-\, \divv_x\big( E(\Lambda,v) \nabla_x\Lambda \big)\, w\, \Theta(\alpha^2 w)\,  +\,
    \left(\frac{1}{\alpha} \,+\,\mu_1\right)w\,-\,\frac{1}{\alpha}
    \xi(v)v\ .
    \end{split}
    \eqnn
Setting $k(t):= 1/\alpha^2 +\Xi t/\alpha\ge 1/\alpha^2$ for $t\ge 0$,
we infer from $(h_1)$, \eqref{17}, and the properties of $\Theta$
that $\mathcal{L}_1 k\ge k'-\xi(v)v/\alpha\ge 0$. Therefore,
$u_1(t,x)\le k(t)$ for $(t,x)\in [0,t^+)\times\bar{\Om}$ by the
comparison principle since we assumed that $u_1(0,x)\le
1/\alpha^2$ for $x\in \Om$. Furthermore, we have $
\mathcal{L}_2 u_2=0$, where $\mathcal{L}_2$ is the parabolic operator
defined by
     \bqnn
    \begin{split}
    \mathcal{L}_2\,w \,:=\, \partial_t
    &w\,-\,\divv_x\big(D(\Lambda)\nabla_x w\big)\, -\,
    \big(\alpha^2
    u_2\Theta'(\alpha^2 u_2)\,+\,\Theta(\alpha^2 u_2)\big) E(\Lambda,v)\nabla_x\Lambda\cdot \nabla_x w\\
    &-\, \divv_x\big( E(\Lambda,v) \nabla_x\Lambda \big)\, w\, \Theta(\alpha^2 w)\,  +\,
    \left(\frac{1}{\alpha} \,+\,\mu_2\right)w\,-\,\frac{1}{\alpha}
    u_1\ .
    \end{split}
    \eqnn
Owing to the previous bound on $u_1$ and the properties of
$\Theta$, we have $\mathcal{L}_2 k\ge k'+ (k-u_1)/\alpha\ge 0$.
Using again the comparison principle we conclude that $u_2(t,x)\le
k(t)$ for $(t,x)\in [0,t^+)\times\bar{\Om}$. Proceeding
analogously for $u_i$, $i\in\{3,\ldots,I\}$, we derive that
    \bqn\label{21}
    0\le u_i(t,x)\le k(t)\ ,\qquad (t,x)\in [0,t^+)\times\bar{\Om}\
    ,\quad i\in\{1,\ldots,I\}\ .
    \eqn
Now, by $(h_1)$, \eqref{13}, and \eqref{21} we clearly have
$$
\partial_t v\,-\,\alpha \Delta_x v\,\le\, \|g\|_\infty\, v \,+\,
\alpha\, k(t) \sum_{i=1}^I b_i\, \mu_i
$$
from which we deduce that
 \bqn\label{22}
    0\,\le\, v(t,x)\,\le\, \|v^0\|_\infty\, e^{\|g\|_\infty t}\, +\, \alpha \sum_{i=1}^I b_i \,\mu_i\int_0^t k(s)\, e^{\|g\|_\infty (t-s)}\, \rd s \ ,\qquad (t,x)\in  [0,t^+)\times\bar{\Om}\ .
    \eqn
Finally, by $(h_1)$ and \eqref{21} we have
\bqnn\begin{split}
-\lambda_{I+1}\, k(t)\, - \,\alpha\sum_{i=1}^I \big(\vert\lambda_i^*\vert \,+\,\mu_i\,
\lambda_i\big)\, k(t)\,&\le\, \lambda_1 \,\xi(v)\,v\,+\,
\sum_{i=1}^I \alpha \big(\lambda_i^*\,-\,\mu_i\, \lambda_i\big)\,
u_i\, -\, \lambda_{I+1}\, u_I\\
&\le\, \lambda_1\, \Xi\,+\,\alpha \sum_{i=1}^I
\vert\lambda_i^*\vert\, k(t)\,,
\end{split}
\eqnn
so the comparison principle applied to \eqref{12} warrants that
    \bqn\label{23}
    \vert\Lambda(t,x)\vert\,\le\, \| \Lambda^0\|_\infty\, +\, \lambda_1\, \Xi\, t\,+\, \left[ \lambda_{I+1}+\alpha\sum_{i=1}^I\big(\vert \lambda_i^*\vert \,+\,\mu_i\,\lambda_i\big)\right]\int_0^t k(s)\ \rd s
\eqn
for $(t,x)\in [0,t^+)\times\bar{\Om}$. Thanks to \eqref{21},
\eqref{22}, \eqref{23}, and the upper triangular structure of the
diffusion matrix $a$, we are in a position to apply \cite[Thm.15.5]{Amann93} and conclude that indeed $t^+=\infty$.

It then remains to check \eqref{20}. First note that
$t_\alpha^*>0$ due to \eqref{18} and the continuity of
$(u_1,\ldots,u_I)$. Next, setting
$P:=\alpha\sum_{i=1}^I\lambda_i u_i$, it follows from
\eqref{10} that $P$ solves
$$
\partial_t P+\sum_{i=1}^I \lambda_i
(u_i-u_{i-1})\,=\,
\divv_x\left(D(\Lambda)\nabla_xP+\alpha\sum_{i=1}^I \lambda_i
u_i\Theta(\alpha^2 u_i) E(\Lambda,v) \nabla_x \Lambda \right)- \alpha
\sum_{i=1}^I\lambda_i \mu_i u_i\ .
$$
On one hand, we clearly have
$$
\sum_{i=1}^I\lambda_i (u_i-u_{i-1})=\sum_{i=1}^I\lambda_i u_i -
\sum_{i=0}^{I-1}\lambda_{i+1} u_i = -\lambda_1 \xi(v) v -\alpha
\sum_{i=1}^I\lambda_i^* u_i +\lambda_{I+1} u_I
$$
by \eqref{14} and the definition of $\lambda_i^*$. On the other
hand, if $t\in [0,t_\alpha^*)$, then $\Theta(\alpha^2 u_i(t,x))=1$
for $x\in \bar{\Om}$ and so $\alpha\sum_{i=1}^I\lambda_i u_i
\Theta(\alpha^2 u_i)=P$ in $[0,t_\alpha^*)\times\bar{\Om}$.
Consequently, $P$ solves the same initial-boundary value problem
\eqref{12}, \eqref{15}, \eqref{16} as $\Lambda$ in
$(0,t_\alpha^*)\times\Om$. The uniqueness of classical solutions
to this problem guarantees that \eqref{20} holds true. Thus the
proof of Proposition~\ref{P2.1} is complete.
\end{proof}

\subsection{Uniform Bounds}

The aim of this subsection is to derive some uniform bounds on the
solution obtained in the previous subsection. Given $p>N$ and 
$(u_1^0,\ldots,u_I^0, v^0)\in W_p^1(\Om,\R^{I+1})$ obeying
\eqref{18}, let $(u_1,\ldots,u_I,\Lambda,v)$ denote the classical
solution to \eqref{10}-\eqref{16} provided by Proposition~\ref{P2.1}.

We fix a constant $K_0$ such that
\bear
\nonumber
\int_\Om \left( \alpha \sum_{i=1}^I b_i u_i^0(x) + v^0(x) + \alpha \sum_{i=1}^I \lambda_i \left[ u_i^0(x) (\ln{u_i^0(x)} - 1) +1 \right] \right)\ \rd x & &  \\
\label{wildschwein}
+ b_1 + \lambda_1 +  \left\| \Lambda^0\right\|_\infty + \left\| v^0\right\|_\infty & \le & K_0\,.
 \eear
In the following, $c$ and $c_j$, $j\ge 1$, are generic constants that may differ
from place to place and depend on $\ell$, $B$, $L$, $M$, $\beta$ in \eqref{8}, $\|g\|_\infty$, and $K_0$, but not on $I$, $D$, $E$, $\alpha\in (0,1)$, $d_0$, and $\Xi$ in \eqref{17}. Dependence on additional variables will be indicated explicitly. 

\medskip

We start with an $L_1$-estimate:

\begin{lem}\label{L2.2}
For $T>0$, we have
    \bqn\label{24}
    \int_\Om\left(\alpha \sum_{i=1}^I b_i\, u_i (t,x) +
    v(t,x)\right) \rd x\, \le\, c_1(T)\ ,\quad t\in [0,T]\ .
    \eqn
\end{lem}

\begin{proof} Multiplying \eqref{10} by $\alpha b_i$, summing with respect to $i$, and integrating with respect to $x$, we obtain from \eqref{14} and \eqref{15} the identity
$$
\frac{\rd }{\rd t} \int_\Om \alpha \sum_{i=1}^I b_i u_i \,\rd
x -b_1\int_\Om\xi(v)v\,\rd
x-\int_\Om\alpha\sum_{i=1}^Ib_i^*\,u_i\, \rd x\,=\,- \int_\Om
\left( \alpha\sum_{i=1}^I b_i\,\mu_i\,u_i + b_{I+1}\ u_I \right)\,\rd x\ .
$$
Integrating \eqref{13} with respect to $x$ and adding the result
to the above identity gives
    \bqnn
    \begin{split}
    \frac{\rd }{\rd t}\int_\Om\left(\alpha \sum_{i=1}^I b_i \,u_i
    + v\right) \rd x\,&\le\,\int_\Om \big(g(v)-(1-b_1)\xi(v)\big)v\, \rd
    x+\int_\Om\alpha\sum_{i=1}^I b_i^* u_i\,\rd x\\
    &\le\, b_1\, \|g\|_\infty \int_\Om v\,\rd
    x+\left(\sup_{1\le j\le I} \frac{b_j^*}{b_j} \right)\  \int_\Om \alpha
    \sum_{i=1}^{I} b_i\, u_i\,\rd x
    \end{split}
    \eqnn
due to $(h_1)$, whence the claim from \eqref{8}.
\end{proof}

We next improve the previous $L_1$-estimate. An appropriate choice of the sequence $(\eta_i)_{i\ge 1}$ considered in the forthcoming lemma will allow us to control the tail of the approximating sequence $(u_i)$ in $L_1((0,T)\times (0,\infty)\times\Om;b(a)\rd t\rd a\rd x)$ later on. Indeed, such a property in turn will guarantee one of the two conditions required for the $L_1$-weak compactness of the approximation (see Section~\ref{S4.3} below).

\begin{lem}\label{highlight}
Let $(\eta_i)\inÊ[0,1]^{I+1}$ be such that $\eta_1=0$ and $\eta_i\le \eta_{i+1}$ for $i=1,\ldots,I$. Then, for $T>0$, 
$$
\int_\Om \alpha \sum_{i=1}^I \eta_i\,b_i \,u_i(t,x)\, \rd x \le e^{Bt}\ \int_\Om \alpha \sum_{i=1}^I \eta_i\,b_i \,u_i^0(x)\, \rd x \,+\, \left| \eta^*\right|_\infty\ c_2(T)\ , \quad t\in [0,T]\, ,
$$
where $\eta_i^*:=(\eta_{i+1}-\eta_i)/\alpha$ for $i=1,\ldots,I$ and $\left| \eta^*\right|_\infty := \max_{1\le i\le I} \left| \eta_i^*\right|$.
\end{lem}

\begin{proof}
Multiplying \eqref{10} by $\alpha \eta_i b_i$, summing with respect to $i$, and integrating with respect to $x$, we obtain from \eqref{14} and \eqref{15} the inequality
$$
\frac{\rd }{\rd t} \int_\Om \alpha \sum_{i=1}^I \eta_i b_i u_i \,\rd
x - \int_\Om\alpha\sum_{i=1}^I \left( b_{i+1}\ \frac{\eta_{i+1}-\eta_i}{\alpha} + \eta_i\ b_i^* \right)\,u_i\, \rd x\,\le\,0\ .
$$
Owing to \eqref{8} we have $b_{i+1}=b_i+\alpha\ b_i^*\le (1+\alpha B)\ b_i$, so that 
\bean
\frac{\rd }{\rd t} \int_\Om \alpha \sum_{i=1}^I \eta_i b_i u_i \,\rd
x & \,\le\, & \int_\Om\alpha\sum_{i=1}^I \left( (1+\alpha B)\ \eta_i^* + \eta_i\ B \right)\,b_i \,u_i\, \rd x \\
& \,\le\, & B\ \int_\Om\alpha\sum_{i=1}^I \eta_i\,b_i \,u_i\, \rd x + (1+ B)\, \left|\eta^*\right|_\infty\, c_1(T)\ ,
\eean
the last inequality being a consequence of \eqref{24}. The claim then follows by integration. 
\end{proof}

We next derive $L_\infty$-estimates on $\Lambda$ and $v$. To this end recall that $t_\alpha^*>0$ was defined in \eqref{19}.

\begin{lem}\label{L2.3}
For $T>0$, we have
    \bqn
    \label{25}
    \|\Lambda (t)\|_\infty + \|v(t)\|_\infty \le c_3(T)\ ,\qquad t\in [0,T]\cap [0,t_\alpha^*)\ .
    \eqn
\end{lem}

\begin{proof}
We first infer from $(h_1)$, \eqref{8}, \eqref{13}, and \eqref{20} that
$$
\partial_t v - \alpha\ \Delta_x v \le \|g\|_\infty\ v + \beta\ \Lambda \;\;\mbox{ in }\;\; (0,t_\alpha^*)\times\Om
$$
with homogeneous Neumann boundary conditions. The comparison principle readily entails that 
$$
\|v(t)\|_\infty \le \left( \left\| v^0\right\|_\infty + \beta\ \int_0^t \|\Lambda(s)\|_\infty\ \rd s \right)\ e^{\|g\|_\infty t}\,, \quad t\in [0,t_\alpha^*)\,.
$$
Next observe that, by $(h_1)$, \eqref{8}, \eqref{12}, and
\eqref{20}, we have
$$
\partial_t \Lambda - \divv_x\left([D(\Lambda)+\Lambda
    E(\Lambda,v)]\nabla_x\Lambda\right) \,\le\, \lambda_1 \|g\|_\infty v + L\ \Lambda \;\;\mbox{ in }\;\; (0,t_\alpha^*)\times\Om\, .
$$ 
Using the comparison principle once more, we deduce that 
$$
\|\Lambda(t)\|_\infty \le \left( \left\| \Lambda^0\right\|_\infty + \lambda_1 \|g\|_\infty\ \int_0^t \|v(s)\|_\infty\ \rd s \right)\ e^{L t}\,, \quad t\in [0,t_\alpha^*)\,.
$$
Consequently, 
$$
\|\Lambda(t)\|_\infty + \|v(t)\|_\infty \le c(T)\, \left( 1 + \int_0^t \left( \|\Lambda(s)\|_\infty + \|v(s)\|_\infty \right) \rd s \right)\,, \quad t\in [0,T]\cap [0,t_\alpha^*)\, ,
$$
from which the claim then follows.
\end{proof}

As a consequence of the preceding lemma we obtain a lower bound
for $t_\alpha^*$.

\begin{cor}\label{C2.7}
Consider $T>0$. If $\alpha\le \ell/(4c_3(T))$, then $t_\alpha^*\,\ge\, T$.
\end{cor}

\begin{proof}
We consider $T>0$ and $\alpha\le \ell/(4c_3(T))$. Assume for contradiction that $t_\alpha^*<T$. Then it follows from \eqref{8}, \eqref{20}, \eqref{25}, the non-negativity and continuity of $u_i$, and the choice of $\alpha$ that
$$\alpha\, \ell\, u_i(t_\alpha^*,x)\le \Lambda(t_\alpha^*,x)\le c_3(T) \le \frac{\ell}{4\alpha}\ ,\quad x\in\Om\ ,\quad i=1,\ldots,I\ ,
$$
whence $u_i(t_\alpha^*,x)\le 1/(4\alpha^2)$ for $x\in\Om$ and $i=1,\ldots,I$.  Together with the continuity of the $u_i$'s, this contradicts the definition \eqref{19} of $t_\alpha^*$.
\end{proof}

Next we establish some bounds on $(u_i)$ that will guarantee its local weak compactness in $L_1$ and the strong compactness with respect to space and time of its averages with respect to the age variable. As already mentioned, this approach is inspired by the existence proof of renormalized solutions to the Boltzmann equation \cite{DPL89} which makes use of velocity averaging results, see, e.g., \cite{GLPS88}, \cite[Chapter~5]{Pe02}, and the references therein.

\begin{lem}\label{L2.5}
For any $T>0$,
    \bear
    \label{29a}
    \sum_{i=1}^I \alpha\,\lambda_i\int_\Om \phi\big(u_i(t,x)\big)\, \rd x & \,\le\, & c_4(T)\,, \\
        \label{29b}
\int_0^t\int_\Om \left( \alpha \sum_{i=1}^I\lambda_i\,  D(\Lambda)\, \big\vert \nabla_x u_i^{1/2}\big\vert^2 + E(\Lambda,v)\, \big\vert\nabla_x \Lambda\big\vert^2 \right)\, \rd x\rd s & \, \le\, & c_4(T)\ , 
\eear
for $t\in [0,T]\cap [0,t_\alpha^*)$, where $\phi(r):=r(\ln r-1)+1$ for $r>0$ and $\phi(0):=1$. In addition,
\bqn
\label{29c}
\int_0^t\int_\Om \left( \big\vert\nabla_x \zeta_1(\Lambda)\big\vert^2 + \big\vert\nabla_x \zeta_2(\Lambda)\big\vert^2 \right)\, \rd x\rd s \, \le\, c_5(T)\ ,\quad t\in [0,T]\cap [0,t_\alpha^*)\ ,
\eqn
where $\zeta_1$ and $\zeta_2$ are defined in $(h_5)$ and $(h_6)$, respectively. 
\end{lem}

\begin{proof}
Multiplying \eqref{10} by $\alpha\lambda_i\ln u_i$, summing the
resulting equations with respect to $i$, and integrating over
$\Om$, we obtain
    \bqnn
    \begin{split}
    \frac{\rd}{\rd t} \sum_{i=1}^I  \alpha\int_\Om \lambda_i\, \phi(u_i)\,\rd
    x=& -\sum_{i=1}^I\int_\Om \lambda_i\,(u_i-u_{i-1})\ln u_i\, \rd x -
    \alpha\sum_{i=1}^I\int_\Om  \lambda_i\,\mu_i\, u_i\, \ln u_i\, \rd x\\
    & -\sum_{i=1}^I\alpha\,\lambda_i\,\int_\Om \left\{
    D(\Lambda)\frac{\vert\nabla_x u_i\vert^2}{u_i}+E(\Lambda,v)\ \nabla_x u_i\cdot\nabla_x \Lambda\right\} \, \rd x\ .
    \end{split}
    \eqnn
Taking into account that
$$
    -r\ln r\,\le\, \phi(r)+\,r\ ,\quad r\ge 0\ ,
$$
and that
$$
    (u_i-u_{i-1})\ln u_i\ge \phi(u_i)-\phi(u_{i-1})
$$
due to the convexity of $\phi$, we derive from
 \eqref{14} and \eqref{20}
    \bqnn
    \begin{split}
    \frac{\rd}{\rd t} \sum_{i=1}^I  \alpha\,\lambda_i\int_\Om \phi(u_i)\,\rd
    x&\,\le\,  - \sum_{i=1}^I\lambda_i\int_\Om \big(\phi(u_i)-\phi(u_{i-1})\big)\, \rd x
    + \sum_{i=1}^I\int_\Om\alpha\,
    \lambda_i\,\mu_i\,\big(\phi(u_i)+u_i\big)\, \rd x\\
    &\qquad - \int_\Om \left\{ \sum_{i=1}^I\alpha\,\lambda_i
    D(\Lambda)\frac{\vert\nabla_x u_i\vert^2}{u_i}\,+\, E(\Lambda,v)\ \big| \nabla_x \Lambda\big|^2\right\}\, \rd x\\
    &\,\le\, \lambda_1\int_\Om\phi(\xi(v)v)\,\rd x -\lambda_{I+1} \int_\Om \phi(u_{I})\,\rd x \,+
    \,\alpha\sum_{i=1}^I\int_\Om\lambda_i^*\,\phi(u_i)\, \rd x\\
      &\qquad + \sum_{i=1}^I\int_\Om\alpha\,
    \lambda_i\,\mu_i\,\big(\phi(u_i)+u_i\big)\, \rd x\\
&\qquad - \int_\Om \left\{ \sum_{i=1}^I\alpha\,\lambda_i
    D(\Lambda)\frac{\vert\nabla_x u_i\vert^2}{u_i}\,+\, E(\Lambda,v)\ \big| \nabla_x \Lambda\big|^2\right\}\, \rd x\ .
     \end{split}
    \eqnn
Recalling that $\phi\ge 0$, $(h_1)$, \eqref{8}, and \eqref{25} we deduce that
    \bqnn
    \begin{split}
    \frac{\rd}{\rd t} \sum_{i=1}^I  \alpha\,\lambda_i\int_\Om \phi(u_i)\,\rd
    x&\,\le \, c(T) \,+\,\alpha \,(L+M)\sum_{i=1}^I\int_\Om\lambda_i\,\phi(u_i)\, \rd x \\
&\qquad - \int_\Om \left\{ \sum_{i=1}^I\alpha\,\lambda_i
    D(\Lambda)\frac{\vert\nabla_x u_i\vert^2}{u_i}\,+\, E(\Lambda,v)\ \big| \nabla_x \Lambda\big|^2\right\}\, \rd x\ ,
    \end{split}
    \eqnn
from which \eqref{29a} and \eqref{29b} follow. 

Next, by the Cauchy-Schwarz inequality and \eqref{20} we have
$$
\left|\nabla_x \Lambda\right| \le 2\ \Lambda^{1/2}\ \left(\sum_{i=1}^I \alpha \lambda_i \big|\nabla_x u_i^{1/2}\big|^2 \right)^{1/2}\, .
$$
From this, $(h_5)$, and \eqref{29b}, we deduce that 
\bean
\int_0^t \int_\Om \big\vert\nabla_x \zeta_1(\Lambda)\big\vert^2\ \rd x\rd s & = & \int_0^t \int_\Om \frac{D(\Lambda)}{\Lambda}\ \big\vert\nabla_x\Lambda\big\vert^2\ \rd x\rd s \\ 
& \le & 4\ \int_0^t \int_\Om D(\Lambda)\ \sum_{i=1}^I \alpha \lambda_i \big|\nabla_x u_i^{1/2}\big|^2\ \rd x\rd s \le c_4(T)\, ,
\eean
while $(h_6)$, \eqref{25}, and \eqref{29b} imply that
$$
\int_0^t \int_\Om \big\vert\nabla_x \zeta_2(\Lambda)\big\vert^2\ \rd x\rd s \le \int_0^t \int_\Om \frac{E(\Lambda,v)}{\kappa_2(c_3(T))}\ \big\vert\nabla_x\Lambda\big\vert^2\ \rd x\rd s \le \frac{c_4(T)}{\kappa_2(c_3(T))}\, ,
$$
thus completing the proof. \end{proof}

\begin{lem}\label{Leasy}
Consider $(\chi_i)\in\R^{I+1}$ with $\chi_{I+1}=0$ and put $\mcm_\chi := \alpha \sum_{i=1}^I \chi_i u_i$  and $\chi_i^*:=(\chi_{i+1}-\chi_i)/\alpha$ for $i=1,\ldots, I$. Then, for $t\in [0,T]\cap [0, t_\alpha^*)$ with $T>0$, we have
\bqn
\label{30a}
\int_0^t \left( \left\| \nabla_x \mcd_\chi(\mcm_\chi) \right\|_2^2 + \left\| \partial_t \mcd_\chi(\mcm_\chi) \right\|_{(W_{N+1}^1(\Om))'} \right)\ \rd s \le c_6\left( T,|\chi|_\infty,|\chi^*|_\infty, G(T) \right)\ ,
\eqn 
where 
$$
G(T):=\|D\|_{L_\infty(0,c_3(T))}+\|E\|_{L_\infty((0,c_3(T))^2)} \, , 
$$
and $\mcd_\chi\in\mathcal{C}^2(\R)$ is any function satisfying $\mcd_\chi(0)=\mcd_\chi'(0)=0$ and 
$$
0 \le \mcd_\chi''(r) \le D\left( \frac{\ell\ |r|}{|\chi|_\infty} \right)\,, \qquad r\in\R\, .
$$
\end{lem}

\begin{proof}
First observe that the Cauchy-Schwarz inequality together with \eqref{8} and \eqref{20} give that 
\bean
\left|\nabla_x \mcm_\chi\right|  & \le & 2\ \alpha \sum_{i=1}^I \left|\chi_i\right|\ u_i^{1/2}\ \left|\nabla_x u_i^{1/2} \right| \le \frac{2\ \alpha |\chi|_\infty}{\ell}\ \sum_{i=1}^I \lambda_i\ u_i^{1/2}\ \left|\nabla_x u_i^{1/2} \right|\\
& \le &  \frac{2\ |\chi|_\infty}{\ell}\ \Lambda^{1/2}\ \left( \sum_{i=1}^I \alpha\ \lambda_i\ \left|\nabla_x u_i^{1/2} \right|^2 \right)^{1/2}\, .
\eean
Consequently, by \eqref{25} and \eqref{29b}, 
\bear
\nonumber
\int_0^t \int_\Om D(\Lambda)\ \left|\nabla_x \mcm_\chi\right|^2\ \rd x\rd s & \le & \frac{4\ |\chi|_\infty^2}{\ell^2}\ \int_0^t \|\Lambda(s)\|_\infty\ \int_\Om \sum_{i=1}^I \alpha\ \lambda_i\ D(\Lambda)\ \left|\nabla_x u_i^{1/2} \right|^2\ \rd x\rd s \\
\label{31a}
& \le & c(T)\ |\chi|_\infty^2\,.
\eear
Owing to \eqref{8} we have 
\bqn\label{M}
|\mcm_\chi|\le |\chi|_\infty\ \Lambda/\ell\ ,
\eqn
and the monotonicity of $D$ warrants that
\bqn
\label{31b}
\left|\mcd_\chi'(\mcm_\chi)\right| \le \int_0^{|\mcm_\chi|} D\left( \frac{\ell\ r}{|\chi|_\infty} \right)\ \rd r \le |\mcm_\chi|\ D\left( \frac{\ell\ |\mcm_\chi|}{|\chi|_\infty} \right)\le |\mcm_\chi|\ D(\Lambda)\le \frac{|\chi|_\infty}{\ell}\ \Lambda\ D(\Lambda)\,.
\eqn
We then infer from \eqref{25}, \eqref{31a}, and \eqref{31b} that 
\bean
\int_0^t \int_\Om \left| \nabla_x \mcd_\chi(\mcm_\chi) \right|^2\ \rd x\rd s & \le & \frac{|\chi|_\infty^2}{\ell^2}\ \int_0^t \int_\Om \Lambda^2\ D(\Lambda)^2\ \left|\nabla_x \mcm_\chi\right|^2\ \rd x\rd s \\
& \le & c(T)\ |\chi|_\infty^2 \,,
\eean
whence the first part of \eqref{30a}. 

Next, we infer from \eqref{10} that $\mcm_\chi$ solves 
$$
\partial_t \mcm_\chi = \divv_x\big( D(\Lambda)\ \nabla_x\mcm_\chi + \mcm_\chi\ E(\Lambda,v)\ \nabla_x\Lambda \big) + \mcm_{\chi^*-\chi \mu} + \chi_1\ \xi(v)\ v \;\;\mbox{ in }\;\; (0,t_\alpha^*)\times\Om
$$
with homogeneous Neumann boundary conditions. Consider $\varphi\in W_{N+1}^1(\Om)$. Multiplying the above equation by $\varphi\ \mcd_\chi'(\mcm_\chi)$ and integrating over $\Om$ give
\bqn
\label{32a}
\int_0^t \langle \partial_t \mcd_\chi(\mcm_\chi(s)) , \varphi \rangle_{W_{N+1}^1(\Om)}\ \rd s = F_1 + F_2 + F_3\, ,
\eqn
where
\bean
F_1 & := & \int_0^t \int_\Om \big( D(\Lambda)\ \nabla_x \mcm_\chi + \mcm_\chi\ E(\Lambda,v)\ \nabla_x\Lambda \big)\ \mcd_\chi'(\mcm_\chi)\ \nabla_x\varphi\ \rd x\rd s\,, \\
F_2 & := & \int_0^t \int_\Om \big( D(\Lambda)\ \nabla_x \mcm_\chi + \mcm_\chi\ E(\Lambda,v)\ \nabla_x\Lambda \big)\ \mcd_\chi''(\mcm_\chi)\ \nabla_x\mcm_\chi\ \varphi\ \rd x\rd s\,, \\
F_3 & := & \int_0^t \int_\Om \big( \mcm_{\chi^*-\chi\mu} + \chi_1\ \xi(v)\ v \big)\ \mcd_\chi'(\mcm_\chi)\ \varphi\ \rd x\rd s\,.
\eean
It first follows from \eqref{25}, \eqref{29b}, \eqref{31a}, \eqref{M}, and \eqref{31b} that
\bean
|F_1| & \le & \int_0^t \left\| D(\Lambda)^{1/2}\ \nabla_x \mcm_\chi \right\|_2\ \left( \int_\Om D(\Lambda)\ \mcd_\chi'(\mcm_\chi)^2\ |\nabla_x\varphi|^2 \rd x \right)^{1/2} \ \rd s\\
& & +\ \int_0^t \left\| E(\Lambda,v)^{1/2}\ \nabla_x \Lambda \right\|_2\ \left( \int_\Om E(\Lambda,v)\ \mcm_\chi^2\ \mcd_\chi'(\mcm_\chi)^2\ |\nabla_x\varphi|^2 \rd x \right)^{1/2}\ \rd s \\
& \le & c(T,|\chi|_\infty)\ \|\nabla_x\varphi\|_{N+1}\,.
\eean
Next, \eqref{M}, the assumption on $\mcd_\chi''$, and the monotonicity of $D$ yield
$$
\mcd_\chi''(\mcm_\chi) \le D\left( \frac{\ell \vert\mcm_\chi\vert}{|\chi|_\infty}\right) \le D(\Lambda)\,.
$$
We then similarly infer from \eqref{25}, \eqref{29b}, \eqref{31a}, and the continuous embedding of $W_{N+1}^1(\Om)$ in $L_\infty(\Om)$ that
\bean
|F_2| & \le & \|\varphi\|_\infty\ \int_0^t \int_\Om D(\Lambda)\ \mcd_\chi''(\mcm_\chi)\ |\nabla_x\mcm_\chi|^2\ \rd x \rd s\\
& & +\  \|\varphi\|_\infty\ \int_0^t \left\| E(\Lambda,v)^{1/2}\ \nabla_x \Lambda \right\|_2\ \left( \int_\Om E(\Lambda,v)\ \mcm_\chi^2\ \mcd_\chi''(\mcm_\chi)^2\ |\nabla_x\mcm_\chi|^2 \rd x \right)^{1/2}\ \rd s \\
& \le & c\, \|\varphi\|_{W_{N+1}^1(\Om)}\ \int_0^t \int_\Om D(\Lambda)^2\ |\nabla_x\mcm_\chi|^2\ \rd x \rd s\\
& & +\, c\,\|\varphi\|_{W_{N+1}^1(\Om)}\ \left( \int_0^t \int_\Om E(\Lambda,v)\ \mcm_\chi^2\ D(\Lambda)^2\ |\nabla_x\mcm_\chi|^2 \rd x \rd s \right)^{1/2} \\
& \le & c(T,|\chi|_\infty)\ \|\varphi\|_{W_{N+1}^1(\Om)}\,.
\eean
Finally, we deduce from $(h_1)$, \eqref{8}, \eqref{25}, and \eqref{31b} that
\bean
|F_3| & \le & \int_0^t \int_\Om |\varphi|\ \frac{|\chi|_\infty}{\ell}\ \Lambda\ D(\Lambda)\ \left( |\chi|_\infty\ \|g\|_\infty\ v + \frac{(|\chi^*|_\infty+M\ |\chi|_\infty)}{\ell}\ \Lambda \right)\ \rd x\rd t\\
& \le & c(T,|\chi|_\infty,|\chi^*|_\infty)\ \|\varphi\|_{W_{N+1}^1(\Om)}\,.
\eean
Combining the above three estimates for $F_1$, $F_2$, and $F_3$ with \eqref{32a} leads to the assertion by a duality argument.
\end{proof}

We finish off this section with an estimate on $\Delta_x v$.

\begin{lem}\label{L2.4}
For any $T>0$, we have    
    \bqn\label{28}
    \alpha^2\int_0^t \|\Delta_x v\|_2^2\, \rd s\,\le \,\alpha\,
    \|\nabla_x v^0\|_2^2 +c(T)\ , \qquad t\in [0,T]\cap [0,t_\alpha^*)\ .
    \eqn
\end{lem}

\begin{proof}
We multiply \eqref{13} by $-\alpha\Delta_x v$, integrate over $\Om$, and use $(h_1)$, \eqref{8}, and \eqref{25} to obtain
    \bqnn
    \begin{split}
    \frac{\alpha}{2}\frac{\rd}{\rd t}\|\nabla_x v\|_2^2 +\alpha^2
    \|\Delta_x v\|_2^2& \le \alpha \|g\|_\infty\, \|v\|_2\, \|\Delta_x
    v\|_2 + \alpha \beta \|\Lambda\|_2\,\|\Delta_x v\|_2\\ &
    \le\frac{\alpha^2 }{2} \|\Delta_x v\|_2^2 + c(T)
    \end{split}
    \eqnn
in $[0,T]\cap [0,t_\alpha^*)$. Integrating with respect to time yields \eqref{28}.
\end{proof}

\section{Passing to the Limit}

We are now in a position to prove Theorem~\ref{T}. We thus
construct a weak solution to \eqref{1}-\eqref{6} in the sense of
Definition~\ref{D} by using a compactness argument for the
solution to the regularized problem \eqref{10}-\eqref{16}. First
we demonstrate how we set up equations \eqref{10}-\eqref{16}.

\subsection{Approximation}

Suppose hypotheses $(h_1)-(h_6)$. Choose $\alpha \in (0,1)$ and
put $\Ia:=\left[ 1/\alpha^2 \right]$. We then set
$$
D_\alpha(r):= D(r)+\alpha\ ,\quad r\in \R\ ,
$$
and note that $D_\alpha$ satisfies \eqref{17} with $d_0=\alpha$. By classical approximation arguments, we also construct a non-negative function $E_\alpha\in\mathcal{C}^3(\R^2)\cap L_\infty(\R^2)$ such that $E_\alpha(r,s)=E(r,s)$ for $(r,s)\in [0,1/\alpha]^2$ and a non-negative function $\xi_\alpha\in\mathcal{C}^1(\R)\cap L_\infty(\R)$ such that $s\mapsto (1+s)\xi_\alpha(s)$ belongs to $L_\infty(\R)$ and $\xi_\alpha(s)=\xi(s)$ for $s\in [0,1/\alpha]$. 
Furthermore, we set, for $i=1,\ldots,\Ia+1$,
$$
\li:=\frac{1}{\alpha}\int_{(i-1)\alpha}^{i\alpha} \lambda(a)\,\rd
a\ ,\quad \bi:=\frac{1}{\alpha}\int_{(i-1)\alpha}^{i\alpha}
b(a)\,\rd a\ ,\quad
\mi:=\frac{1}{\alpha}\int_{(i-1)\alpha}^{i\alpha} \mu(a)\,\rd a
$$
and, for $i=1,\ldots,\Ia$,
$$
    \lis:=\frac{\lambda_{i+1,\alpha}-\li}{\alpha}\ ,\qquad
    \bis:=\frac{b_{i+1,\alpha}-\bi}{\alpha}\ .
$$
Observe then that $(h_2)-(h_4)$ imply the validity of \eqref{8} with $B=B_0$, $L=L_0$, $\beta=\beta_0(1+B_0)$, and $\ell=\ell_0$.
Indeed, hypothesis $(h_2)$ ensures that
$$
0 \le \bis\,=\,\frac{1}{\alpha}\int_{(i-1)\alpha}^{i\alpha}\frac{
b(a+\alpha)-b(a)}{\alpha}\, \rd a\,\le\, B_0\,\bi\ ,\quad
i=1,\ldots, \Ia\ .
$$
One shows $\lis\le L_0\li$ analogously using $(h_3)$ which also gives $\li\ge\ell_0$.
Finally observe that $(h_4)$ warrants
$$
   \mi\,\bi\,=\,\frac{1}{\alpha^2}\int_{(i-1)\alpha}^{i\alpha}\mu(a)\int_{(i-1)\alpha}^{i\alpha}b(z)\,\rd
    z\rd
    a\,\le\,\frac{\beta_0}{\alpha^2}\int_{(i-1)\alpha}^{i\alpha}\lambda(a)\int_{(i-1)\alpha}^{i\alpha}
    \frac{b(z)}{b(a)}\,\rd z\rd a
$$
and therefore $\mi\, \bi\le \beta_0 (B_0+1)\,\li$ owing to
$$
\frac{b(z)}{b(a)}\,\le\,
\frac{b(i\alpha)}{b\big((i-1)\alpha)\big)}\,\le\, B_0\,\alpha +1\
,\qquad (i-1)\alpha\le a, z\le i\alpha\ ,
$$
since $b$ is non-decreasing. Given $p>N$ and any non-negative
initial values $(u^0,v^0)$ satisfying \eqref{id1} and \eqref{id2},
let $\big(u_{1,\alpha},\ldots,
u_{\Ia,\alpha},\Lambda_\alpha,v_\alpha\big)$ denote 
the classical solution to \eqref{10}-\eqref{16} with $I=\Ia$,
$(\lambda_i,b_i,\mu_i)=(\li,\bi,\mi)$, and where $(D,E,\xi)$ are replaced
by $(D_\alpha,E_\alpha,\xi_\alpha)$ and $u_i^0$ by
$$
u_{i,\alpha}^{0}(x):=\frac{1}{\alpha}\int_{(i-1)\alpha}^{i\alpha}
u^0(a,x)\,\rd a\ ,\qquad x\in\bar{\Om}\ ,\quad i=1,\ldots,\Ia\ .
$$
Note that we may assume without loss of generality that
$\|u_{i,\alpha}^{0}\|_\infty\le 1/4\alpha^2$ by making $\alpha$
smaller if necessary. We first collect some properties of $(u_{i,\alpha}^0,\Lambda_\alpha^0)$.

\begin{lem}\label{L3.0}
For $\alpha>0$ small enough, we have
$$
\int_\Om \alpha \sum_{i=1}^{\Ia} \left( \bi\ u_{i,\alpha}^{0}(x) + \li\ \phi\left( u_{i,\alpha}^{0}(x) \right) \right)\ \rd x + \left\| \Lambda_\alpha^0 \right\|_\infty \le c_7\, ,
$$
where $\phi(r):=r\, (\ln{r}-1)+1$ for $r>0$ and $\phi(0):=1$. 
\end{lem}

\begin{proof}
By $(h_2)$,
\bean
\int_\Om \alpha \sum_{i=1}^{\Ia} \bi\ u_{i,\alpha}^{0}(x)\ \rd x & = & \frac{1}{\alpha}\ \int_\Om \sum_{i=1}^{\Ia} \int_{(i-1)\alpha}^{i\alpha} u^0(a,x)\, \int_{(i-1)\alpha}^{i\alpha} b(z)\ \rd z \rd a \rd x \\
& \le & \frac{1}{\alpha}\ \int_\Om \sum_{i=1}^{\Ia} \int_{(i-1)\alpha}^{i\alpha} b(a)\, u^0(a,x)\, \int_{(i-1)\alpha}^{i\alpha} \frac{b(a+\alpha)}{b(a)}\ \rd z \rd a \rd x \\
& \le & (1+B_0)\ \int_\Om \int_0^\infty b(a)\, u^0(a,x)\ \rd a\rd x\,.
\eean

Next, $(h_3)$, \eqref{id2}, and Jensen's inequality entail that
\bqnn
\begin{split}
\int_\Om \alpha \sum_{i=1}^{\Ia} & \li\ \phi\left( u_{i,\alpha}^{0}(x) \right) \ \rd x \\
& =  \int_\Om \sum_{i=1}^{\Ia} \left( \int_{(i-1)\alpha}^{i\alpha} \lambda(z)\ \rd z \right)\ \phi\left( \frac{1}{\alpha} \int_{(i-1)\alpha}^{i\alpha} u^{0}(a,x)\ \rd a \right) \ \rd x \\
& \le  \frac{1}{\alpha}\ \int_\Om \sum_{i=1}^{\Ia} \int_{(i-1)\alpha}^{i\alpha} \lambda(z)\ \int_{(i-1)\alpha}^{i\alpha} \phi\left( u^{0}(a,x) \right)\ \rd a \rd z\rd x \\
& =  \frac{1}{\alpha}\ \int_\Om \sum_{i=1}^{\Ia} \int_{(i-1)\alpha}^{i\alpha} \phi\left( u^{0}(a,x) \right)\ \int_{(i-1)\alpha}^a \lambda(z)\ \rd z \rd a\rd x \\
& \quad +\ \frac{1}{\alpha}\ \int_\Om \sum_{i=1}^{\Ia} \int_{(i-1)\alpha}^{i\alpha} \phi\left( u^{0}(a,x) \right)\ \int_a^{i\alpha} \lambda(z)\ \rd z \rd a\rd x \\
& \le  \frac{1}{\alpha}\ \int_\Om \sum_{i=1}^{\Ia} \int_{(i-1)\alpha}^{i\alpha} \lambda(a)\, \phi\left( u^{0}(a,x) \right)\, \left(\int_{(i-1)\alpha}^a (1+L_0(a-z))\ \rd z  \right) \rd a\rd x \\
&\quad  +\ \frac{1}{\alpha}\ \int_\Om \sum_{i=1}^{\Ia} \int_{(i-1)\alpha}^{i\alpha} \lambda(a)\, \phi\left( u^{0}(a,x) \right)\,
\left( \int_a^{i\alpha}(1+L_0(z-a))\ \rd z  \right)
 \rd a\rd x \\
& \le  (1+L_0)\ \int_\Om \sum_{i=1}^{\Ia} \int_{(i-1)\alpha}^{i\alpha} \lambda(a)\ \phi\left( u^{0}(a,x) \right)\ \rd a\rd x \\
& \le  (1+L_0)\ \int_\Om \int_0^\infty \lambda(a)\ \phi\left( u^{0}(a,x) \right)\ \rd a\rd x \, .
\end{split}
\eqnn

Finally, as above we deduce 
\bean
0\le \Lambda_\alpha^0(x) & = & \frac{1}{\alpha}\ \sum_{i=1}^{\Ia} \int_{(i-1)\alpha}^{i\alpha} u^0(a,x)\ \left( \int_{(i-1)\alpha}^a \lambda(z)\ \rd z + \int_a^{i\alpha} \lambda(z)\ \rd z  \right)\ \rd a \\
& \le & (1+L_0)\ \int_0^\infty \lambda(a)\ u^0(a,x)\ \rd a\,,
\eean
and the right-hand side of the above inequality belongs to  $L_\infty(\Om)$ as a consequence of \eqref{id1} and the continuous embedding of $W_p^1(\Om)$ in $L_\infty(\Om)$ (recall that $p>N$). 
\end{proof}

Now we introduce
    \begin{align*}
&\lambda_\alpha(a):=\sum_{i=1}^{\Ia} \li {\bf
1}_{((i-1)\alpha,i\alpha]}(a)\ ,& b_\alpha(a):=\sum_{i=1}^{\Ia}
\bi {\bf 1}_{((i-1)\alpha,i\alpha]}(a)\ ,\\ &
\mu_\alpha(a):=\sum_{i=1}^{\Ia} \mi {\bf
1}_{((i-1)\alpha,i\alpha]}(a)\ ,
    \end{align*}
and
$$
u_\alpha(t,a,x):= \sum_{i=1}^{\Ia} u_{i,\alpha}(t,x)\, {\bf
1}_{((i-1)\alpha,i\alpha]}(a)
$$
for $(t,a,x)\in [0,\infty)\times [0,\infty)\times\Om$. Let
$$
    t_\alpha^*:=\sup\left\{t>0\,;\, \max_{1\le i\le I} \sup_{\tau\in [0,t]} \|u_{i,\alpha}(\tau)\|_\infty  \le
    \frac{1}{2\alpha^2}\right\}\,>0\ ,
$$
be defined as in \eqref{19}. We first establish that, as expected, $u_\alpha$ is a weak solution to an approximation of the original problem. 

\begin{lem}\label{L3.1}
If $\varphi\in \mathcal{C}^1\big([0,\infty)\times\bar{\Om}\big)$ is such that
$\mathrm{supp}\, \varphi\subset [0,R]\times\bar{\Om}$ for some
$R>0$ and $\alpha\Ia\ge R$, then
    \bqnn
    \begin{split}
    \frac{\rd}{\rd t}\int_0^\infty\int_\Om \varphi \, u_\alpha \, \rd
    a\,\rd x  =\frac{1}{\alpha}&\int_0^\alpha\int_\Om
    \varphi \,\xi_\alpha(v_\alpha )\, v_\alpha \,\rd x\rd
    a\\
    & +\int_0^\infty\int_\Om
    \left(\frac{\varphi(a+\alpha)-\varphi(a)}{\alpha}-\varphi(a)\mu_\alpha(a)\right)
    u_\alpha (a)\, \rd x\rd a\\
    & -\int_0^\infty\int_\Om
    \big(D_\alpha(\Lambda_\alpha)\nabla_x u_\alpha
    +u_\alpha\, E_\alpha(\Lambda_\alpha , v_\alpha )\ \nabla_x \Lambda_\alpha \big)\cdot \nabla_x\vp\,\rd x\rd a
    \end{split}
    \eqnn
for $t\in (0,t_\alpha^*)$.
\end{lem}

\begin{proof}
Let $\varphi\in \mathcal{C}^1\big([0,\infty)\times\bar{\Om}\big)$ with
$\mathrm{supp}\, \varphi\subset [0,R]\times\bar{\Om}$ and put
$$
\vp_{i,\alpha}(x):=\frac{1}{\alpha}\int_{(i-1)\alpha}^{i\alpha}\vp(a,x)\,\rd
a\ ,\qquad i=1,\ldots,\Ia+1\ ,\quad x\in\Om\ .
$$
We infer from \eqref{10} and \eqref{14} that, for $t\in
(0,t_\alpha^*)$,
\bqnn
\begin{split}
 \frac{\rd}{\rd t}&\int_0^\infty\int_\Om \varphi \, u_\alpha\, \rd
    a\rd x  = \frac{\rd}{\rd t}\sum_{i=1}^{\Ia} \int_\Om \alpha\,\varphi_{i,\alpha} \, u_{i,\alpha}\,\rd x\\
    &=\sum_{i=1}^{\Ia} \int_\Om \alpha\,\varphi_{i,\alpha}\left[-\frac{1}{\alpha}\big(u_{i,\alpha}-u_{i-1,\alpha}\big)-\mu_{i,\alpha}u_{i,\alpha}\right] \rd x\\
&\quad+\ \sum_{i=1}^{\Ia} \int_\Om \alpha\,\varphi_{i,\alpha} \,\divv_x\big(D_\alpha(\Lambda_\alpha)\nabla_x  u_{i,\alpha}+u_{i,\alpha}
    E_\alpha(\Lambda_\alpha,v_\alpha)\, \nabla_x \Lambda_\alpha \big) \rd x\\ 
    & = \int_\Om\vp_{1,\alpha}\, \xi_\alpha(v_\alpha)\,v_\alpha\,\rd
    x+\sum_{i=1}^{\Ia}\int_\Om\left(\vp_{i+1,\alpha}-\vp_{i,\alpha}\right)
    u_{i,\alpha}\,\rd x -\int_\Om\vp_{\Ia+1,\alpha}\,
    u_{\Ia,\alpha}\,\rd x\\
 &   \quad -\ \int_0^\infty\int_\Om\vp\, \mu_\alpha \,u_\alpha\, \rd
    x\rd a
    -\int_0^\infty\int_\Om\big(D_\alpha(\Lambda_\alpha)\nabla_x
    u_\alpha +u_\alpha\nabla_x
    E_\alpha(\Lambda_\alpha,v_\alpha)\nabla_x \Lambda_\alpha\big)\cdot\nabla_x\vp\, \rd x\rd a\ .
\end{split}
\eqnn
Noticing that $\alpha\Ia\ge R$ implies $\vp_{\Ia+1,\alpha}=0$, the
assertion follows.
\end{proof}

\subsection{Compactness estimates}

Our aim is then to pass to the limit as $\alpha\to 0$ in the identity stated in the previous lemma. We thus need to provide some compactness for $(u_\alpha)$, $(\Lambda_\alpha)$, and $(v_\alpha)$, a first step being the derivation of suitable estimates. 

\begin{lem}\label{L3.3}
Let $T>0$. Then, for $\alpha$ small enough (depending on $T$), $A\ge 4$, and $t\in [0,T]$, the following estimates are valid:
    \begin{align}
    \int_0^\infty\int_\Om \left[ b(a)\, u_\alpha(t,a,x) + \lambda(a)\, \phi\big(u_\alpha(t,a,x)\big) \right] \rd x\rd a &\le c_1(T)+c_4(T)\ ,\label{30}\\
    \|\Lambda_\alpha(t)\|_\infty\,+\,\|v_\alpha(t)\|_\infty &\le c_3(T)\ ,\label{31}\\
    \int_0^t \int_\Om \left( \left|\nabla_x \zeta_1(\Lambda_\alpha) \right|^2\ + \left\vert \nabla_x \zeta_2(\Lambda_\alpha) \right\vert^2 + \alpha\  \left\vert \nabla_x \Lambda_\alpha \right\vert^2 \right)\, \rd x\rd s&\le c_8(T)\ , 
    \label{32}
        \end{align}
        \bqn
\int_\Om \int_A^\infty b(a)\, u_\alpha(t,a,x)\, \rd a\rd x \le c_8(T)\ \left( \int_\Om \int_{A/4}^\infty b(a) \,u^0(a,x)\, \rd a \rd x \,+\, \frac{1}{A} \right)\ .
\label{34}
\eqn
In addition, for any $\chi\in \mathcal{C}^1([0,\infty))$ with compact support, the sequence $(\mcm_{\chi,\alpha})$ defined by 
$$
\mcm_{\chi,\alpha}(t,x) := \int_0^\infty \chi(a)\ u_\alpha(t,a,x)\ \rd a\,, \qquad (t,x)\in (0,t_\alpha^*)\times\Om\,,
$$
is such that 
\bqn
\label{33}
\left\| \mcm_{\chi,\alpha}(t)\right\|_\infty + \int_0^t \left( \left\| \nabla_x \mcd_\chi(\mcm_{\chi,\alpha}) \right\|_2^2 + \left\| \partial_t \mcd_\chi(\mcm_{\chi,\alpha}) \right\|_{(W_{N+1}^1(\Om))'} \right)\ \rd s \le c_9(T,\|\chi\|_{W_\infty^1(0,\infty)})
\eqn 
for $t\in [0,T]\cap [0,t_\alpha^*)$, where the function $\mcd_\chi\in\mathcal{C}^2(\R)$ is defined by $\mcd_\chi(0)=\mcd_\chi'(0):=0$ and 
$$
\mcd_\chi''(r) := D\left( \frac{\ell\ |r|}{\|\chi\|_\infty} \right) \le D_\alpha\left( \frac{\ell\ |r|}{\|\chi\|_\infty} \right)\,, \qquad r\in\R\, .
$$
\end{lem}

\begin{proof}
Owing to Lemma \ref{L3.0} and assumption $(h_4)$, the condition \eqref{wildschwein} is fulfilled with
$$
K_0\,:=\, c_7\,+\, (1+\vert\Om\vert)\, \| v^0\|_\infty\, +\, (1+\beta_0)\, b(2)\ ,
$$
which clearly does not depend on the approximation parameter $\alpha$.

Let $T>0$. According to Corollary~\ref{C2.7}, we may choose $\alpha$ small enough (depending on $T$) such that $t_\alpha^*>T$. We can also assume that $\alpha$ satisfies $\alpha<1/c_3(T)$, the constant $c_3(T)$ stemming from in Lemma~\ref{L2.3}. 

Observe first that \eqref{30} and \eqref{31} are immediate consequences of Lemmata~\ref{L2.2}, \ref{L2.3}, and \ref{L2.5}. A useful consequence of \eqref{31} and the choice $\alpha<1/c_3(T)$ is that $\xi_\alpha(v_\alpha)=\xi(v_\alpha)$ and $E_\alpha(\Lambda_\alpha,v_\alpha)=E(\Lambda_\alpha,v_\alpha)$. We then infer from \eqref{29c} that 
$$
\int_0^t\int_\Om \left( \big\vert\nabla_x \zeta_{1,\alpha}(\Lambda_\alpha)\big\vert^2 + \big\vert\nabla_x \zeta_2(\Lambda_\alpha)\big\vert^2 \right)\, \rd x\rd s \, \le\, c_5(T)\ ,\quad t\in [0,T]\, ,
$$
with $\zeta_{1,\alpha}'(r):=((D(r)+\alpha)/r)^{1/2}$ and $\zeta_{1,\alpha}(0) := 0$. Clearly, $\zeta_{1,\alpha}'(\Lambda_\alpha)\ge \zeta_1'(\Lambda_\alpha)$ and also
$$
\zeta_{1,\alpha}'(\Lambda_\alpha) \ge \frac{\alpha^{1/2}}{\Lambda_\alpha^{1/2}} \ge \frac{\alpha^{1/2}}{c_3(T)^{1/2}}
$$
by \eqref{31}. Collecting the above information allows us to conclude that \eqref{32} holds true. Also, introducing
$$
\chi_i:=\frac{1}{\alpha}\ \int_{(i-1)\alpha}^{i\alpha} \chi(a)\ \rd a \quad\mbox{ and }\quad \chi_i^*:= \frac{\chi_{i+1}-\chi_i}{\alpha}\,, \quad i\ge 1\,,
$$
for $\chi\in \mathcal{C}^1([0,\infty))$ with compact support, we readily see that $|\chi|_\infty+|\chi^*|_\infty\le \|\chi\|_{W_\infty^1(0,\infty)}$ and \eqref{33} follows at once from Lemma~\ref{Leasy}, \eqref{31}, and the inequality $\mcm_{\chi,\alpha}\le \|\chi\|_\infty\Lambda_\alpha/\ell$. 

Finally, let $\eta\in \mathcal{C}^\infty(\R)$ be a fixed non-decreasing function such that $\eta(a)=0$ for $a\le 1/2$ and $\eta(a)=1$ for $a\ge 1$. For $A\ge 4$ and $i\in\N\setminus\{0\}$, we put $\eta_i:=\eta(i\alpha/A)$. Then $\eta_1=0$ and $(\eta_i)_{i\ge 1}$ is clearly a non-decreasing sequence with $0\le \eta_i^*:=(\eta_{i+1}-\eta_i)/\alpha\le \|\eta'\|_\infty/A$ for $i\ge 1$. We then infer from Lemma~\ref{highlight} that, for $t\in [0,T]$, 
$$
\int_\Om \alpha \sum_{i=1}^{\Ia} \eta_i\,\bi \,u_{i,\alpha}(t,x)\, \rd x \le e^{B_0t}\ \int_\Om \alpha \sum_{i=1}^{\Ia} \eta_i\,\bi \,u_{i,\alpha}^0(x)\, \rd x \,+\, \frac{\|\eta'\|_\infty}{A}\ c_2(T)\ . 
$$
The properties of $b_\alpha$ and $\eta$ imply that
\bean
\int_\Om \int_A^\infty b(a)\, u_\alpha(t,a,x)\, \rd a\rd x & \le & \int_\Om \int_A^\infty b_\alpha(a)\, u_\alpha(t,a,x)\, \rd a\rd x \\
& \le & \int_\Om \alpha \sum_{i=1}^{\Ia} \eta_i\,\bi \,u_{i,\alpha}(t,x)\, \rd x \\
& \le & e^{B_0t}\ \int_\Om \int_{(A-2)/2}^\infty b_\alpha(a) \,u^0(a,x)\, \rd a \rd x \,+\, \frac{\|\eta'\|_\infty}{A}\ c_2(T)\ , 
\eean
whence \eqref{34} by $(h_2)$, the latter guaranteeing that $b_\alpha(a)\le (1+B_0)b(a)$ for $a>0$.
\end{proof}

\begin{rem}
Owing to the superlinearity of $\phi$ at infinity, the estimate \eqref{30} warrants the weak compactness of $(u_\alpha)$ in $L_1((0,T)\times (0,A)\times \Om)$ for $T>0$ and $A>0$ and strongly relies on the assumed positivity of $\lambda$. Therefore, if $\lambda$ would vanish on some interval $(0,a_0)$ with $a_0>0$ (as, e.g., in \cite{Ayati1, Ayati2, EsipovShapiro, MMNP}), the restriction of $(u_\alpha)$ to the set $(0,T)\times (0,a_0)\times \Om$ is only weakly-$*$ compact in the space of bounded measures. In that case the passage to the limit performed in the next section might be more delicate.
\end{rem}

\subsection{Proof of Theorem~\ref{T}}\label{S4.3}

\begin{proof} We fix $T>0$. Due to the positive lower bound on $\lambda_\alpha$, \eqref{30}, and \eqref{34}, we may apply the Dunford-Pettis theorem \cite[IV.8]{DunfordSchwartz} to conclude that there are a sequence $(\alpha_k)_{k\ge 1}$ with $\alpha_k\to 0$ and a non-negative function $u\in L_1((0,T)\times (0,\infty)\times \Om; b(a) \rd t \rd a \rd x)$  such that
\bqn
\label{39}
u_{\alpha_k} \rightharpoonup u \;\;\mbox{ in }\;\; L_1((0,T)\times (0,\infty)\times \Om; b(a) \rd t \rd a \rd x)\,.
\eqn
We may also assume that, for each $k\ge 1$, $\alpha_k$ is small enough such that $t_{\alpha_k}^*\ge T$ and $c_3(T)<1/\alpha_k$. Next, given any $\chi\in\mathcal{C}^1([0,\infty))$ with compact support, we readily deduce from \eqref{39} and the positivity and unboundedness of $b$ that 
\bqn
\label{40}
\mcm_{\chi,\alpha_k} \rightharpoonup \mcm_\chi \;\;\mbox{ in }\;\; L_1((0,T)\times\Om)\,,
\eqn
where $\mcm_{\chi,\alpha_k}$ is defined in Lemma~\ref{L3.3} and 
$$
\mcm_\chi(t,x) := \int_0^\infty \chi(a)\ u(t,a,x)\ \rd a\,, \quad (t,x)\in (0,T)\times\Om\,.
$$
Furthermore, owing to \eqref{33}, we may apply \cite[Corollary~4]{Si87} to conclude that $\left( \mcd_\chi\left(\mcm_{\chi,\alpha_k}\right) \right)_k$ is relatively (strongly) compact in $L_2((0,T)\times\Om)$, hence converges also a.e. (after a possible extraction of a further subsequence). This property, the strict monotonicity of $\mcd_\chi$, and \eqref{33} then imply that
\bqn
\label{41}
\mcm_{\chi,\alpha_k} \longrightarrow \mcm_\chi \;\;\mbox{ in }\;\; L_q((0,T)\times\Om) \;\;\mbox{ for any }\;\; q\in [1,\infty)\,.
\eqn

We next claim that 
\bear
\label{42}
\Lambda_{\alpha_k} &\longrightarrow& \Lambda:=\mcm_\lambda \;\;\mbox{ in }\;\; L_q((0,T)\times\Om) \;\;\mbox{ and a.e. in }\;\; (0,T)\times\Om\,, \\
\label{43}
\mcm_{b_{\alpha_k} \mu_{\alpha_k},\alpha_k} &\longrightarrow& \mcm_{b \mu} \;\;\mbox{ in }\;\; L_q((0,T)\times\Om)\,,
\eear
for any $q\in [1,\infty)$. Indeed, let $\vartheta\in \mathcal{C}^\infty(\R)$ be a smooth and non-increasing cut-off function satisfying $\vartheta(a)=1$ if $a\le 1/2$ and $\vartheta(a)=0$ if $a\ge 1$. For $A\ge 1$ and $a\ge 0$, we put $\vartheta_A(a):=\vartheta(a/2A)$. We infer from $(h_4)$ and \eqref{34} that 
\bean
\int_0^T \int_\Om \left| \Lambda_{\alpha_k} - \Lambda\right|\ \rd x\rd t & = & \int_0^T \int_\Om \left| \int_0^\infty \lambda(a)\ \left( u_{\alpha_k} - u \right)\ \rd a \right|\ \rd x \rd t \\
& \le & \int_0^T \int_\Om \left| \int_0^\infty \lambda(a)\ \vartheta_A(a)\ \left( u_{\alpha_k} - u \right)\ \rd a \right|\ \rd x \rd t \\
& & +\ \beta_0\ \int_0^T \int_\Om \int_A^\infty b(a)\ \left( u_{\alpha_k} + u \right)\ \rd a\rd x \rd t\\
& \le & \int_0^T \int_\Om \left| \mcm_{\lambda \vartheta_A,\alpha_k} - \mcm_{\lambda \vartheta_A} \right|\ \rd x \rd t \\
& & +\ c(T)\ \left( \int_0^T \int_\Om \int_{A/4}^\infty b(a)\ \left( u^0 + u \right)\ \rd a\rd x \rd t + \frac{1}{A} \right)\,,
\eean
whence
$$
\limsup_{k\to\infty} \int_0^T \int_\Om \left| \Lambda_{\alpha_k} - \Lambda\right|\ \rd x\rd t \le c(T)\ \left( \int_0^T \int_\Om \int_{A/4}^\infty b(a)\ \left( u^0 + u \right)\ \rd a\rd x \rd t + \frac{1}{A} \right)
$$
by \eqref{41}. Letting $A\to\infty$ completes the proof of \eqref{42} for $q=1$. The extension to $q\in (1,\infty)$ next follows from \eqref{31} by interpolation. The proof of \eqref{43} is similar and uses additionally $(h_2)$. 

A further consequence of \eqref{41} is that, for any $\varphi\in\mathcal{C}^1([0,T)\times[0,\infty)\times\bar{\Om})$ with compact support and $q\in [1,\infty)$, we have 
\bqn
\label{41b}
\mcm_{\varphi,\alpha_k} \longrightarrow \mcm_\varphi \;\;\mbox{ in }\;\; L_q((0,T)\times\Om) 
\eqn
with the notations
$$
\mcm_{\varphi,\alpha_k}(t,x) := \int_0^\infty \varphi(t,a,x)\ u_{\alpha_k}(t,a,x)\ \rd a \;\;\mbox{ and }\;\; \mcm_\varphi(t,x) := \int_0^\infty \varphi(t,a,x)\ u(t,a,x)\ \rd a
$$
for $(t,x)\in (0,T)\times\Om$. Indeed, we argue as in \cite[Section~IV]{DPL89} and first note that \eqref{41b} readily follows from \eqref{41} if 
there are an integer $J\ge 1$ and functions $(\psi_j)_{1\le j \le J}$ in $\mathcal{C}^1([0,T)\times\bar{\Om})$ and $(\chi_j)_{1\le j \le J}$ in $\mathcal{C}^1([0,\infty))$ with compact support such that 
\bqn
\label{41c}
\varphi(t,a,x)=\sum_{j=1}^J \psi_j(t,x)\ \chi_j(a) \;\;\mbox{ for }\;\; (t,a,x)\in [0,T)\times[0,\infty)\times\bar{\Om}\,.
\eqn
We next use the classical fact that, given $\varphi\in\mathcal{C}^1([0,T)\times[0,\infty)\times\bar{\Om})$ with compact support, there is a sequence of functions $(\varphi_n)_n$ which is bounded in $L_\infty((0,T)\times(0,\infty)\times\Om)$ and converges a.e. towards $\varphi$, each function $\varphi_n$ being of the form \eqref{41c}. The claim \eqref{41b} then follows with the help of the convergence \eqref{39}.

We next turn to the (strong) compactness of $(v_{\alpha_k})_k$ and let $v$ denote the solution to 
\bear
\label{45}
\partial_t v(t,x) & = & (g-\xi)(v(t,x))\ v(t,x) + \mcm_{b\mu}(t,x)\,, \quad (t,x)\in (0,T)\times\Om\, , \\
\label{46}
v(0,x) & = & v^0(x)\,, \quad x\in\Om\,.
\eear
Owing to $(h_1)$, $(h_4)$, and the non-negativity of $u$ and $v^0$, we have $v\ge 0$ and $\partial_t v \le \|g\|_\infty\ v + \beta_0\ \Lambda$. Since $\Lambda$ belongs to $L_\infty((0,T)\times\Om)$ by \eqref{31} and \eqref{42}, so does $v$ by the previous differential inequality. It next follows from $(h_1)$, \eqref{13}, \eqref{31}, and \eqref{45} that 
\bear
\nonumber
\frac{1}{2}\ \frac{\rd}{\rd t} \left\| v_{\alpha_k} - v \right\|_2^2 & \le & - \alpha_k\ \left\| \nabla v_{\alpha_k} \right\|_2^2 - \alpha_k\ \int_\Om \Delta v_{\alpha_k}\ v\ \rd x + \left\| v_{\alpha_k} - v \right\|_2\ \left\| \mcm_{b_{\alpha_k} \mu_{\alpha_k},\alpha_k} - \mcm_{b\mu} \right\|_2
 \\
 \nonumber
& & +\ \int_\Om \left| (g-\xi)(v_{\alpha_k})\ v_{\alpha_k} - (g-\xi)(v)\ v\right|\ \left| v_{\alpha_k} - v \right|\ \rd x \\
\nonumber
& \le & - \alpha_k\ \int_\Om \Delta v_{\alpha_k}\ v\ \rd x + \left\| \mcm_{b_{\alpha_k} \mu_{\alpha_k},\alpha_k} - \mcm_{b\mu} \right\|_2^2 \\
\label{47}
& & +\ \left( 1+ \|g\|_\infty + c_3(T)\ \|g'-\xi'\|_{L_\infty(0,c_3(T))} \right)\ \left\| v_{\alpha_k} - v \right\|_2^2\,.
\eear
Recalling \eqref{28}, we realize that $(\alpha_k \Delta v_{\alpha_k})_k$ is weakly relatively compact in $L_2((0,T)\times\Om)$. Since $(\alpha_k v_{\alpha_k})_k$ converges to zero in $L_2((0,T)\times\Om)$ by \eqref{31}, we thus conclude that $(\alpha_k \Delta v_{\alpha_k})_k$ converges weakly to zero in $L_2((0,T)\times\Om)$. Consequently, as $v$ belongs to $L_\infty((0,T)\times\Om)$, we have
\bqn
\label{48}
\lim_{k\to\infty} \alpha_k\ \int_0^T \int_\Om \Delta v_{\alpha_k}\ v\ \rd x\rd t = 0\,. 
\eqn
We then infer from \eqref{43}, \eqref{47}, and \eqref{48} that
\bqn
\label{49}
v_{\alpha_k} \longrightarrow v \;\;\mbox{ in }\;\; L_2((0,T)\times\Om) \;\;\mbox{ and a.e. in }\;\; (0,T)\times\Om\,,
\eqn
the almost everywhere convergence being obtained after possibly extracting a further subsequence.

We are now in a position to pass to the limit as $\alpha_k\to 0$ in the identity of Lemma~\ref{L3.1} which we first formulate in a different way: if $\varphi\in\mathcal{C}^2([0,T)\times[0,\infty)\times\bar{\Om})$ is compactly supported and satisfies $\partial_\nu\varphi=0$ on $[0,T)\times[0,\infty)\times\partial\Om$, it follows from Lemma~\ref{L3.1} that
\bqn
\label{50}
- \int_0^\infty\int_\Om \varphi(0,a,x) \, u^0_{\alpha_k}(a,x) \, \rd
    a\,\rd x  = \int_0^T \int_0^\infty \int_\Om \partial_t\varphi\ u_{\alpha_k}\ \rd x \rd a \rd t + \sum_{n=1}^4 G_{n,k}(\varphi)
\eqn
with
\bean
G_{1,k}(\varphi) &:= & \int_0^T \int_\Om
\xi(v_{\alpha_k} )\, v_{\alpha_k} \left( \frac{1}{{\alpha_k}}\ \int_0^{\alpha_k}     \varphi \,\rd a \right)\, \rd x\rd t \ ,\\
G_{2,k}(\varphi) &:= & \int_0^T \int_0^\infty\int_\Om \left(\frac{\varphi(t,a+{\alpha_k},x)-\varphi(t,a,x)}{{\alpha_k}}-\varphi(t,a,x)\mu_{\alpha_k}(a)\right)\  u_{\alpha_k} (t,a,x)\, \rd x\rd a\rd t\ ,\\
G_{3,k}(\varphi) &:= & \int_0^T \int_0^\infty \int_\Om \Delta_x\varphi\, D_{\alpha_k}(\Lambda_{\alpha_k})\, u_{\alpha_k}\ \rd x\rd a\rd t\ ,\\
G_{4,k}(\varphi) &:= & \int_0^T \int_\Om \mathbf{J}_{\alpha_k} \cdot  \left( \int_0^\infty u_{\alpha_k}\ \nabla_x \varphi\ \rd a \right) \,\rd x\rd t\ ,\\
\mathbf{J}_{\alpha_k} & := & \nabla_x D_{\alpha_k}(\Lambda_{\alpha_k})\,-\, E(\Lambda_{\alpha_k} , v_{\alpha_k} )\ \nabla_x \Lambda_{\alpha_k} \,.
\eean
First, the boundedness \eqref{31} of $(v_{\alpha_k})_k$, the convergence \eqref{49}, the continuity of $\varphi$, $(h_1)$, and the Lebesgue dominated convergence theorem allow us to pass to the limit in $(G_{1,k}(\varphi))_k$ and conclude that
\bqn
\label{51}
\lim_{k\to\infty} G_{1,k}(\varphi) = \int_0^T \int_\Om
\xi(v(t,x))\, v(t,x)\ \varphi(t,0,x)\, \rd x\rd t\,.
\eqn
In order to handle $(G_{2,k}(\varphi))_k$,  we recall the following consequence of the Dunford-Pettis and Egorov theorems, which is implicitly contained in \cite[p.341]{DPL89} (for a proof see \cite[Lem.A.2]{LauMiArch} for instance).

\begin{lem}\label{leweakae}
Let $U$ be an open bounded subset of $\R^m$, $m\ge 1$, and consider two sequences $(y_k)_k \in L_1(U)$ and $(z_k)_k \in L_\infty(U)$ and functions $y \in L_1(U)$ and $z\in L_\infty(U)$ such that
$$
y_k \rightharpoonup y \;\;\mbox{ in}\;\; L_1(U)\,, \quad |z_k(x)|\le K\;\;\mbox{ and }\;\; \lim_{n\to\infty} z_k(x) = z(x) \;\;\mbox{ a.e. in }\;\; U
$$
for some $K>0$. Then $(y_k z_k)_k$ converges weakly towards $yz$ in $L_1(U)$.
\end{lem}

We now note that 
$$
\lim_{k\to\infty} \left( \frac{\varphi(t,a+{\alpha_k},x)-\varphi(t,a,x)}{{\alpha_k}}-\varphi(t,a,x)\mu_{\alpha_k}(a)\right) = \partial_a \varphi(t,a,x) - \varphi(t,a,x)\ \mu(a)
$$
for a.e. $(t,a,x)\in (0,T)\times(0,\infty)\times\Om$ and is bounded in $(0,T)\times (0,\infty)\times\Om$ by $(h_4)$. Due to this fact and the weak convergence \eqref{39} of $(u_{\alpha_k})_k$ in $L_1$ we may apply Lemma~\ref{leweakae} and deduce that 
\bqn
\label{52}
\lim_{k\to\infty} G_{2,k}(\varphi) = \int_0^T \int_0^\infty\int_\Om
    \big(\partial_a \varphi(t,a,x)-\varphi(t,a,x)\mu(a)\big)\  u(t,a,x)\, \rd x\rd a\rd t\,.
\eqn
Similarly, $\left( \varphi\, D_{\alpha_k}(\Lambda_{\alpha_k} )\right)_k$ is bounded and converges a.e. in $(0,T)\times(0,\infty)\times\Om$ by virtue of \eqref{31} and \eqref{42}, so that the same argument applies to establish that
\bqn
\label{53}
\lim_{k\to\infty} G_{3,k}(\varphi) = \int_0^T \int_0^\infty \int_\Om \Delta_x\varphi\, D(\Lambda)\, u\ \rd x\rd a\rd t\,.
\eqn
Finally, $\mathbf{J}_{\alpha_k}$ also reads 
$$
\mathbf{J}_{\alpha_k} = \frac{D'}{\zeta_1'}(\Lambda_{\alpha_k})\ \nabla_x\zeta_1(\Lambda_{\alpha_k})\,-\, \frac{E(\Lambda_{\alpha_k} , v_{\alpha_k} )}{\zeta_2'(\Lambda_{\alpha_k})} \ \nabla_x\zeta_2(\Lambda_{\alpha_k}) \,,
$$
and we infer from \eqref{32} and \eqref{42} that, after extracting a further subsequence if necessary, we may assume that
$$
\nabla_x\zeta_1(\Lambda_{\alpha_k}) \rightharpoonup \nabla_x\zeta_1(\Lambda) \quad\mbox{ and }\quad \nabla_x\zeta_2(\Lambda_{\alpha_k}) \rightharpoonup \nabla_x\zeta_2(\Lambda) \ \mbox{ in }\ L_2((0,T)\times\Om)\,.
$$
In addition, $(h_5)$ and $(h_6)$ imply the local boundedness of $D'/\zeta_1'$ and $E/\zeta_2'$, so that
$$
\frac{D'}{\zeta_1'}(\Lambda_{\alpha_k}) \longrightarrow \frac{D'}{\zeta_1'}(\Lambda) \;\;\mbox{ and }\;\; \frac{E(\Lambda_{\alpha_k} , v_{\alpha_k} )}{\zeta_2'(\Lambda_{\alpha_k})} \longrightarrow \frac{E(\Lambda , v )}{\zeta_2'(\Lambda)} \;\;\mbox{ in }\;\; L_4((0,T)\times\Om)
$$
by \eqref{31}, \eqref{42}, and \eqref{49}. We then conclude from the above two convergence results that 
$$
\mathbf{J}_{\alpha_k} \rightharpoonup \mathbf{J}:= \frac{D'}{\zeta_1'}(\Lambda)\ \nabla_x\zeta_1(\Lambda)\,-\, \frac{E(\Lambda , v )}{\zeta_2'(\Lambda)} \ \nabla_x\zeta_2(\Lambda) \;\;\mbox{ in }\;\; L_{4/3}((0,T)\times\Om)\,.
$$
Since 
$$
G_{4,k}(\varphi) = \int_0^T \int_\Om \mathbf{J}_{\alpha_k} \cdot  \mcm_{\nabla_x\varphi, \alpha_k} \,\rd x\rd t
$$
and $\mcm_{\nabla_x\varphi, \alpha_k}\longrightarrow \mcm_{\nabla_x\varphi}$ in $L_4((0,T)\times\Om)$ by \eqref{41b}, we finally obtain
\bqn
\label{54}
\lim_{k\to\infty} G_{4,k}(\varphi) = \int_0^T \int_\Om \left( \frac{D'}{\zeta_1'}(\Lambda)\ \nabla_x\zeta_1(\Lambda)\,-\, \frac{E(\Lambda , v )}{\zeta_2'(\Lambda)} \ \nabla_x\zeta_2(\Lambda) \right) \cdot  \left( \int_0^\infty u\ \nabla_x \varphi\, \rd a \right) \,\rd x\rd t\,.
\eqn
Due to \eqref{51}-\eqref{54}, we may pass to the limit as $k\to\infty$ in \eqref{50} and deduce that $u$ solves \eqref{1} in the weak sense stated in Definition~\ref{D}.
\end{proof}

\begin{rem}\label{Miraculix}
It follows from the proof in Section \ref{S4.3} that the strong compactness \eqref{49} of $(v_\alpha)$ and the possibility to pass to the limit in the weak formulation of equation \eqref{1} heavily rely on the strong compactness \eqref{41} of the age averages of $(u_\alpha)$. In turn, this compactness property stems from the positivity of $D$ on $(0,\infty)$. Therefore, the interesting case where swarming is only due to the drift term $\divv_x(uE(\Lambda,v)\nabla_x\Lambda)$ (corresponding to $D\equiv 0$) cannot be handled in a straightforward way. A similar remark applies to the case where $D$ vanishes on a neighborhood of 0 (e.g. as in \cite{Ayati1}).
\end{rem}

\begin{rem}\label{hinkelstein}
It is quite clear from the proof of Theorem~\ref{T} that the regularity assumptions on the data $D$, $b$, $\lambda$, and the initial data $u^0$, $v^0$ are mainly needed to apply the results in \cite{Amann93} for the approximation \eqref{10}-\eqref{16} and thus could be weakened. For instance, it would be sufficient for $v^0$  to be in $L_\infty(\Om)$ instead of $W_p^1(\Om)$. Similarly, one could replace the assumption $u^0 \in L_1\big((0,\infty),W_p^1(\Om);\lambda(a) \rd a \big)$ by $u^0\in L_\infty(\Om;L_1((0,\infty);\lambda(a) \rd a))$. The only modifications to be done in the proof of Theorem~\ref{T} would be the construction of suitable approximations $(D_\alpha)$, $(b_\alpha)$, $(\lambda_\alpha)$, $(u_\alpha^0)$, and $(v_\alpha^0)$ to $D$, $b$, $\lambda$, $u^0$, and $v^0$, which can be done in a classical way.
\end{rem}

\section*{Acknowledgement} We thank the Institut de Math\'ematiques de Toulouse and the Institut f\"ur Angewandte Mathematik of the Leibniz Universit\"at Hannover for hospitality and support during the preparation of this work.

\vspace{1cm}

\end{document}